\documentclass[a4paper,12pt,oneside,reqno]{amsart}
\usepackage{amssymb,amsfonts,amsmath,amsthm,amscd, bbm}
\usepackage{graphicx, color}
\numberwithin{equation}{section}  

\newcommand{\beq}{\begin{equation}} 
\newcommand{\eeq}{\end{equation}} 
\newcommand{\bea}{\begin{aligned}}
\newcommand{\eea}{\end{aligned}}
\newcommand{\bdm}{\begin{displaymath}}
\newcommand{\edm}{\end{displaymath}}
\newcommand{\barr}{\begin{array}}
\newcommand{\earr}{\end{array}}
\newcommand{\ben}{\begin{enumerate}}
\newcommand{\een}{\end{enumerate}}
\newcommand{\bde}{\begin{description}}
\newcommand{\ede}{\end{description}}

\newtheorem{teor}{Theorem}
\newtheorem{prop}[teor]{Proposition}  
\newtheorem{lem}[teor]{Lemma}

\newcommand{\R}{\mathbb{R}}

\newcommand{\N}{\mathbb{N}}

\newcommand{\PP}{\mathbb{P}}
\newcommand{\E}{{\mathbb{E}}}

\newcommand{\defi}{\equiv} 

\newcommand{\ee}{\text{e}}

\newcommand{\w}{\omega}

\newcommand{\vare}{\varepsilon}

\newcommand{\F}{\mathcal{F}}
\newcommand{\1}{\mathbbm{1}}

\begin{document}
\title[Ergodicity of branching Brownian motion]
{An ergodic theorem for the frontier \\ of branching Brownian motion}

\author[L.-P. Arguin]{Louis-Pierre  Arguin}            
 \address{L.-P. Arguin\\ Universit\'{e} de Montr\'{e}al\\ 2920 chemin de la Tour\\
Montr\'{e}al, QC H3T 1J4 \\ 
Canada}
\email{arguinlp@dms.umontreal.ca}
\author[A. Bovier]{Anton Bovier}
\address{A. Bovier\\Institut f\"ur Angewandte Mathematik\\Rheinische
   Friedrich-Wilhelms-Uni\-ver\-si\-t\"at Bonn\\Endenicher Allee 60\\ 53115
   Bonn,Germany}
\email{bovier@uni-bonn.de}

\author[N. Kistler]{Nicola Kistler}
\address{N. Kistler\\Institut f\"ur Angewandte Mathematik\\Rheinische
   Friedrich-Wilhelms-Uni\-ver\-si\-t\"at Bonn\\Endenicher Allee 60\\ 53115
   Bonn,
Germany}
\email{nkistler@uni-bonn.de}


\thanks{A. Bovier is partially supported through the German Research Council in the SFB 611 and
the Hausdorff Center for Mathematics. N. Kistler is partially supported by the Hausdorff Center for Mathematics.}

 \date{\today}

\begin{abstract} 
We prove a conjecture of Lalley and Sellke [{\it Ann. Probab.} {\bf 15} (1987)] asserting that the empirical (time-averaged) distribution function of the maximum of branching Brownian motion 
converges almost surely to a double exponential, or Gumbel, distribution with a random shift.
The method of proof is based on the decorrelation of the maximal displacements for appropriate time scales.
A crucial input is the localization of the paths of particles close to the maximum that was previously established by the authors [{\it Comm. Pure Appl. Math.} {\bf 64} (2011)].
\end{abstract}

\maketitle


\section{Introduction} \label{introduction}
Branching Brownian Motion (BBM) is a continuous-time Markov branching process
 which plays an important role in the theory of partial differential 
equations \cite{aronson_weinberger,aronson_weinberger_two, mckean}, 
in particle physics \cite{munier_peschanski}, 
in the theory of disordered systems
\cite{ BovierKurkova_II, derrida_spohn}, and in mathematical biology \cite{fisher, kessler_et_al}. It is constructed as follows.  
Consider a standard Brownian motion $x(t)$, starting at $0$ at time $0$.
We consider $x(t)$ to be the position of a {\it particle} at time $t$.
After an exponential random time $T$ of mean one and independent of $x$, the particle splits into $k$ particles
with probability $p_k$, where $\sum_{k=1}^\infty p_k = 1$, $\sum_{k=1}^\infty k p_k = 2$, and $\sum_{k} k(k-1) p_k < \infty$. 
The positions of the $k$ particles are independent Brownian motions starting at $x(T)$.
Each of these processes have the same law as the first Brownian particle. 
Thus, after a time $t>0$, there will be $n(t)$ particles  located at $x_1(t), \dots, x_{n(t)}(t)$, with $n(t)$ being the random number of offspring generated up to that time (note that $\E n(t)=e^t$). 

An interesting link between BBM and partial differential equations 
was observed by McKean \cite{mckean}. If one denotes by 
\beq \label{bbm_repr}
u(t, x) \defi \PP\left[ \max_{1\leq k \leq n(t)} x_k(t) \leq x \right]
\eeq
the law of the maximal displacement, a renewal argument shows that 
$u(t,x)$ solves
the Kolmogorov-Petrovsky-Piscounov equation [KPP],  also referred to as the Fisher-KPP equation,
\beq \bea \label{kpp_equation}
& u_t = \frac{1}{2} u_{xx} + \sum_{k=1}^\infty p_k u^k -u, \\
& u(0, x)= 
\begin{cases}
1, \; \text{if}\;  x\geq 0,\\
0, \, \text{if}\; x < 0. 
\end{cases}
\eea \eeq 
This equation has raised a lot of interest, in part because it admits traveling wave solutions:
 there exists a unique solution satisfying 
\beq \label{travelling_one}
u\big(t, m(t)+ x \big) \to \omega(x) \qquad \text{uniformly in}\;  x\; \text{as} \; t\to \infty,
\eeq
with the centering term, the {\it front} of the wave, given by
\beq \label{centering_kpp}
m(t) = \sqrt{2} t - \frac{3}{2\sqrt{2}} \ln t, 
\eeq
and $\w(x)$ is the unique solution (up to translation) of the o.d.e. 
\beq \label{wave_pde}
\frac{1}{2} \omega_{xx} + \sqrt{2} \omega_x + \omega^2 - \omega = 0.
\eeq
The leading order of the front has been established by Kolmogorov, Petrovsky, and Piscounov \cite{kpp}. The logarithmic corrections have been obtained by Bramson \cite{bramson}, using the probabilistic representation given above. 

Equations \eqref{bbm_repr} and \eqref{travelling_one} show the weak convergence of the distribution of the recentered maximum of BBM.

Let
\beq
M(t) \defi \max_{k\leq n(t)} x_k(t)-m(t)\ ,
\eeq
and define for $k= 1 \dots n(t)$,
\beq \label{defi_y}
y_k(t) \defi \sqrt{2}t-x_k(t) \qquad z_k(t) \defi y_k(t)e^{-\sqrt{2}y_k(t)}.
\eeq
With this notation, we consider the quantities 
\beq \label{defi_martingale}
Y(t) \defi \sum_{k\leq n(t)} e^{-\sqrt{2} y_k(t)} \qquad Z(t) \defi \sum_{k\leq n(t)} z_k(t) \ .
\eeq
In 1987, Lalley and Sellke \cite{lalley_sellke} proved that
\beq \label{y_to_zero}
\bea
\lim_{t\uparrow \infty} Y(t) &= 0 \text{ a.s. } \qquad \text{and} \qquad \lim_{t\uparrow \infty} Z(t) &= Z \text{ a.s.},
\eea
\eeq
where $Z$ is a strictly positive random variable with infinite mean. 

This paper is concerned with the large time limit of the empirical (time-averaged) distribution of the maximal displacement 
\beq \label{random_dist}
F_T(x) \defi \frac{1}{T} \int_0^T \1\{M(s) \leq x \} ~ ds, ~ x\in\R
\eeq
The main result is that $F_T$ converges almost surely as $T\to\infty$ to a random distribution function. 
The limit is the double exponential (Gumbel) distribution that is shifted by the random variable $Z$: 

\begin{teor}[Ergodic Theorem]\label{conjecture_compact}
For any $x\in\R$,
\beq \bea \label{goal_two}
\lim_{T\uparrow \infty} F_T(x) = \exp\left( -C Z e^{-\sqrt{2} x}\right) \  \text{ almost surely,}
\eea \eeq
where $C>0$ is a positive constant.
\end{teor}
The derivative martingale $Z$ encodes the dependence on the early evolution of the system. The mechanism 
for this is subtle, and we shall provide first some intuition in the next section. \\

The limit \eqref{goal_two} was first conjectured by Lalley and Sellke in \cite{lalley_sellke}.
They showed that, despite the weak convergence \eqref{travelling_one}, 
the empirical distribution $F_T(x)$ {\it cannot} converge to $\omega(x)$ in the limit of large times (for any $x\in \R$), and proved that the latter is recovered when $Z$ is integrated, i.e.
\beq \label{gumbel_like}
\omega(x) = \E\left[ \exp\left(- C Z \ee^{-\sqrt{2}x} \right)\right].
\eeq
The issue of ergodicity of BBM has also been discussed by Brunet and Derrida in \cite{brunet_derrida_two}.

Ergodic results similar to Theorem \ref{conjecture_compact} can be proved for statistics of extremal particles of BBM other than the distribution of the maximum.
(This will be detailed in a separate work).
Throughout the paper, we use the term {\it extremal} to denote particles at distance of order one from the maximum. 
We also refer to the level of the maximum of the positions as the {\it edge}, or {\it frontier}.

A description of the law of the statistics of extremal particles has been obtained 
in a series of papers of the authors \cite{abk_genealogies, abk_poissonian, abk_extremal} and in the work of A\"{\i}d\'{e}kon, Beresticky, Brunet, and Shi \cite{aidekon_et_al}. 
It is now known that the joint distribution of extremal particles recentered by $m(t)$
converges weakly to a randomly shifted Poisson cluster process. The positions of the clusters is a random shift of a Poisson point process with exponential 
density. The law of the individual clusters is characterized in terms of a branching Brownian motion conditioned to perform unusually large displacements. 
A description of such conditioned BBMs has been given by Chauvin and Rouault \cite{chauvin_rouault}. 

We point out that the interest in the properties of BBM stems also from its alleged universality: it is conjectured, and in some instances also proved,
that different models of probability and of statistical mechanics share many structural features with the extreme values of BBM.  
A partial list includes the two-dimensional Gaussian free field \cite{bdg, bdz, bramson_zeitouni}, the  cover times of graphs by random walks \cite{dembo, dprz},
and in general, log-correlated Gaussian fields, see e.g. \cite{carpentier_ledoussal, bouchaud_fyodorov}.\\

\section{Outline of the proof} \label{conjecture}




It will be convenient to work with compact intervals $\mathcal{D}=[d,D]$ with $-\infty<d<D<\infty$ for the localization procedure 
introduced in Section \ref{intricate}.
Convergence of the empirical distribution on these sets imply convergence of the distribution function $F_T(x)$.
The proof of Theorem \ref{conjecture_compact} goes as follows. 
First, we introduce a "cutoff" $\vare>0$ and split the integration over the sets $[0, T\vare]$ and $(T \vare, T]$. Precisely: with the above notations, we write
\beq
F_T( D ) -F_T(d) = \frac{1}{T} \int_{\vare T}^T \1\{M(s) \in \mathcal D \} ~ ds + \frac{1}{T} \int_{0}^{\vare T} \1\{M(s) \in \mathcal D  \} ~ ds.
\eeq
The second term on the r.h.s. above does not contribute in the limit $T\uparrow \infty$ first and $\vare \downarrow 0$ next. It thus suffices to compute the double limit for the first term. 

To this aim, we introduce the time $R_T>0$, which will play the role of the {\it early evolution}. The precise form is not particularly important and we will specify a choice only later. 
For the moment we only require that $R_T\to \infty$ as $T \uparrow \infty$, but {\it moderately}, i.e. $R_T = o(\sqrt{T})$ in the considered limit of large times. 
We rewrite the empirical distribution as
\beq \bea \label{splitting_vare}
& \frac{1}{T} \int_{\vare T}^T \1\{M(s) \in \mathcal D  \} ds = \\
& \quad =  \frac{1}{T} \int_{\vare T}^T \PP\left[M(s) \in \mathcal D  \mid \mathcal F_{R_T} \right] ds +\\
& \hspace{3cm} + \frac{1}{T}  \int_{\vare T}^T  \Big(  \1\{M(s) \in \mathcal D  \} - \PP\left[M(s) \in \mathcal D  \mid \mathcal F_{R_T} \right] \Big) ds\ .
\eea \eeq

We now state two theorems which immediately imply Theorem \ref{conjecture_compact}: Theorem \ref{convergence_to_der_mart} below addresses the first term 
on the r.h.s of \eqref{splitting_vare}, while Theorem \ref{slln} addresses the second term.

\begin{teor}[Almost sure convergence of the conditional maximum] \label{convergence_to_der_mart} 
Let $R_T \uparrow \infty$ as $T\uparrow \infty$ but with $R_T =o(\sqrt{T})$ in the considered limit. Then for any $s\in [\vare, 1]$, 
\beq
\lim_{T \uparrow \infty} \PP\left[M(T\cdot s) \in \mathcal D  \mid \mathcal F_{R_T} \right] = \int_\mathcal D d \left(\exp \left( -C Z e^{-\sqrt{2} x}\right) \right) \text{ almost surely}.
\eeq
\end{teor}
The above statement is an improvement of \cite[Theorem 1]{lalley_sellke}, where the probability was conditioned on a {\it fixed} time
that only subsequently was let to infinity. The proof closely follows this caseand relies on precise estimates of the law of the maximal displacement obtained by Bramson \cite{bramson_monograph}.

Theorem \ref{convergence_to_der_mart} together with a change of variable and bounded convergence imply
\beq \bea 
& \lim_{\vare \downarrow 0} \lim_{T \uparrow \infty} \frac{1}{T} \int_{\vare T}^T \PP\left[M(s) \in \mathcal D  \mid \mathcal F_{R_T} \right] ds =
 \int_\mathcal D d \left(\exp  -C Z e^{-\sqrt{2} x} \right) \text{ a.s., }
\eea \eeq
which is the r.h.s. of \eqref{goal_two}. 

The integrand of the second term on the r.h.s of \eqref{splitting_vare} has mean zero.
Therefore, Theorem \ref{conjecture_compact} would immediately follow from the above considerations if a strong law of large number holds. This turns out to be correct. 


\begin{teor}[Strong Law of Large Numbers] \label{slln} For $\vare>0$ and $\mathcal D$ as above, 
\beq
\lim_{T\uparrow \infty} \frac{1}{T}  \int_{\vare T}^T  \Big(  \1\{M(s) \in \mathcal D  \} - \PP\left[M(s) \in \mathcal D  \mid \mathcal F_{R_T} \right] \Big) ds = 0 \text{ almost surely.}
\eeq
\end{teor}

Contrary to the case of Theorem \ref{convergence_to_der_mart}, whose short proof is given in Section \ref{easy_prop}, 
the Strong Law of Large Numbers (SLLN) turns out to be quite delicate. Due to the possibly strong correlations among the Brownian particles, 
it is perhaps surprising that a law of large numbers holds at all.
Let $T$ be large and consider two times $s,s'\in[0,T]$. 
It is clear that if the distance between $s$ and $s'$ is of order one, say, then the extremal particles at $s$ are strongly correlated
with the ones at $s'$, since the children of extremal particles are very likely to remain extremal for some time.
Therefore, $s$ and $s'$ need to be {\it well separated} for the correlations to be weak.
On the other hand, and this is the crucial point, it is generally not true 
that the correlations between the extremal particles at time $s$ and $s'$ decay
as the distance between $s$ and $s'$ increases. As shown by Lalley and Sellke \cite[Theorem 2 and corollary]{lalley_sellke}, {\it "every particle born in a branching Brownian motion has a descendant particle in the lead at some future time"}. Hence, if $s$ and $s'$ are too far from each other (for example, if $s$ is of order one with respect to $T$ and $s'$ is of order $T$), correlations build up again and mixing fails. 
Therefore, weak correlations between the frontiers at two different times only set in at precise time scales. It turns out that
if $s$ and $s'$ are both of order $T$, $s,s'\in [\vare T, T]$ and well separated, i.e. $|s-s'|>T^\xi$ for some $0<\xi<1$, then 
the correlations between the frontiers are weak enough to provide a law of large numbers.
By weak enough, we understand a summability condition on the correlations that
lead to a SLLN by a theorem of Lyons, see Theorem \ref{lyons_int} below. See Figure \ref{global_view} a graphical representation. 
A precise control on the correlations is achieved by controlling the paths of extremal particles in the spirit of \cite{abk_genealogies} (see Section \ref{intricate} below for precise statements).

\begin{figure} 
\includegraphics[scale=0.4]{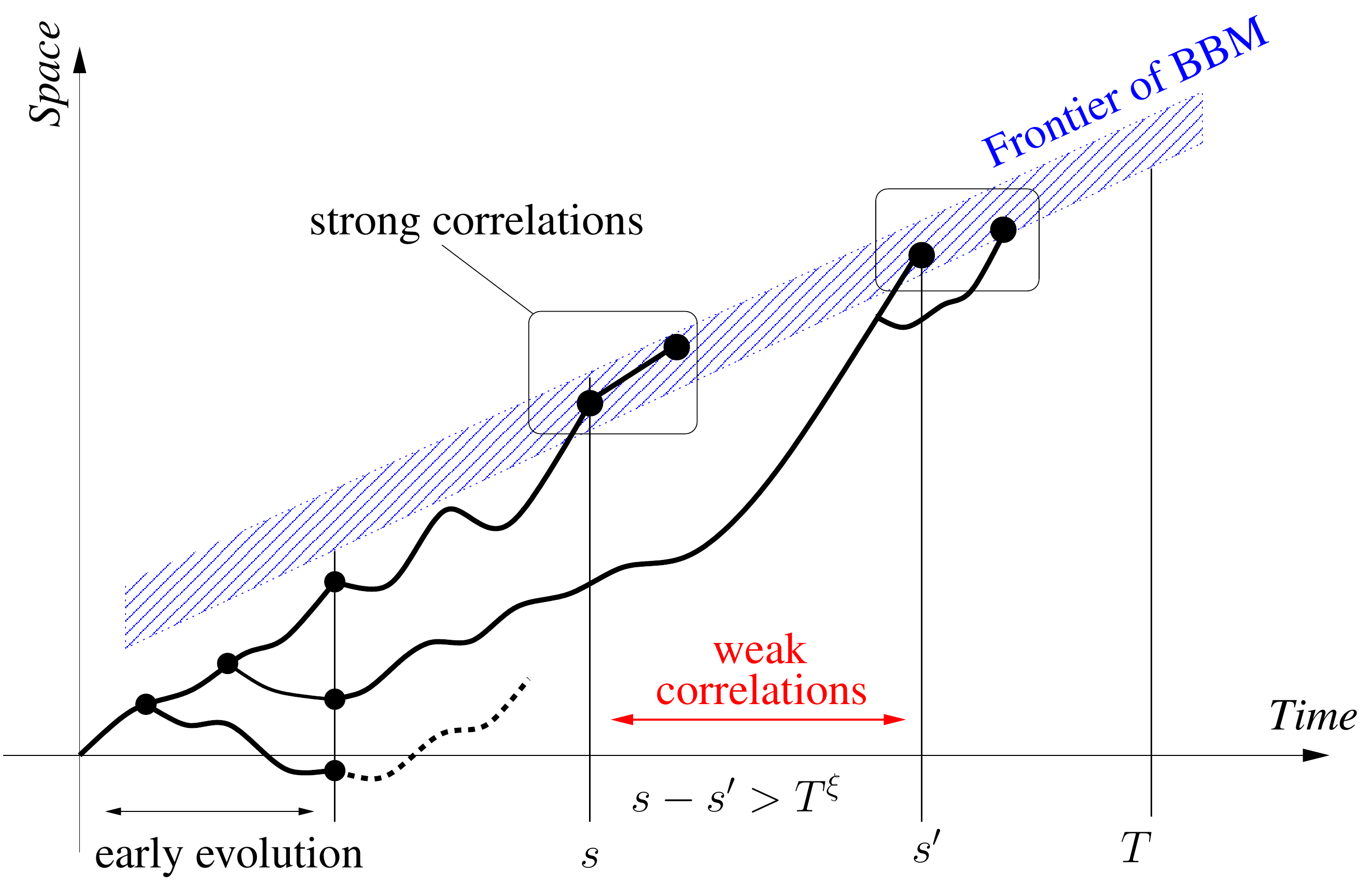}
\caption{Leaders and their ancestors.}
\label{global_view}
\end{figure}

\section{Almost sure convergence of the conditional maximum} \label{easy_prop} 


We start with some elementary facts that will be of importance.
First, observe that for $t,s>0$ such that $s=o(t)$ for $t\uparrow \infty$, the level of the maximum \eqref{centering_kpp} satisfies
\beq \bea \label{splitting_m}
m(t) & = m(t-s)+\sqrt{2} s+ \frac{3}{2\sqrt{2}}\ln\left(\frac{t-s}{t}\right) \\
& = m(t-s)+\sqrt{2} s+  o(1). 
\eea \eeq

Second, let $\{x_j(s), j\leq n(s)\}$  and, for $j=1\dots n(s)$, 
 $\{x_k^{(j)}(t-s), k\leq n^{(j)}(t-s)\}$
be all independent, identically distributed BBMs.  
The Markov property of BBM implies
\beq \label{self_sim}
\{x_k(t), k\leq n(t)\} \stackrel{(d)}{=} \{x_j(s) + x_k^{(j)}(t-s), j\leq n(s),\, k\leq n^{(j)}(t-s)\},
\eeq
In particular, if $\F_s$ denotes the $\sigma$-algebra generated by the process
 up to time $s$, the combinination of \eqref{splitting_m} and \eqref{self_sim} yields for $X\in\R$
\beq \label{equality_cond}
\PP\left[\forall_{k\leq n(t)}:\; \overline{x_k(t)} \leq X \mid \mathcal F_{s} \right] 
= \prod_{k\leq n(s)} \PP\left[\forall_{j\leq n(t-s)}:\; \overline{x_j(t-s)} \leq X + y_k(s) +o(1) \mid \mathcal F_{s} \right] \ .
\eeq
We will typically deal with situations where only a subset of $\{k: \, k=1, \dots, n(t)\}$ appears. In all such cases, the generalization of  \eqref{equality_cond} is straightforward.\\

A key ingredient to the proof of Theorem \ref{convergence_to_der_mart}  is a precise estimate on the right-tail of the distribution of the maximal displacement. 
It is related to \cite[Proposition 3.3]{abk_extremal}, which heavily relies
on the work by Bramson \cite{bramson_monograph}. 

\begin{lem} \label{fundamental_tail}
Consider $t\geq 0$ and $X(t)\geq  0$ such that $\lim_{t\uparrow \infty} X(t) = + \infty$ and
$X(t) = o(\sqrt{t})$ in the considered limit. Then, for $X(t)$ and $t$ both greater than $8r$,
\beq \label{tail_max}
C \gamma(r)^{-1} X(t) e^{-\sqrt{2} X(t)} \left(
1 -\frac{X(t)}{t-r} \right) \leq \PP\left[ M(t) \geq X(t) \right] \leq C \gamma(r) X(t) e^{-\sqrt{2} X(t)}
\eeq
for some $\gamma(r) \downarrow 1$ as $r\to \infty$ and $C$ as in \eqref{gumbel_like}. 
\end{lem}
\begin{proof}
Let us denote by $\overline{u}(t, x) \defi 1 - u(t,x)$, with $u$  the distribution of the maximal displacement defined in  \eqref{bbm_repr}.
We define 
\beq \bea \label{ugly_one}
& \psi(r, t, x+\sqrt{2} t) \defi  \frac{e^{-\sqrt{2}x}}{\sqrt{t-r}} \int_0^\infty \frac{ dy'}{\sqrt{2\pi}} \cdot \overline{u}(r, y'+\sqrt{2}r) \cdot e^{y'\sqrt{2}} \times \\
& \hspace{2cm} \times  \left\{ 1- \exp\left( -2 y'\frac{x+\frac{3}{2\sqrt{2}}\ln t }{t-r} \right) \right\} \exp\left( -\frac{(y'-x)^2}{2(t-r)} \right) .
\eea \eeq
According to \cite[Proposition 3.3]{abk_extremal}, for $r$ large enough, $t\geq 8r$, and $x\geq 8r -\frac{3}{2\sqrt{2}}\ln(t)$, the following bounds
hold: 
\beq \label{ugly_two}
\gamma(r)^{-1}  \psi(r, t, x+\sqrt{2} t) \leq \overline{u}(t, x +\sqrt{2}t) \leq \gamma(r)   \psi(r, t, x+\sqrt{2} t)
\eeq
for some $\gamma(r) \downarrow 1$ as $r\to \infty$. \\

As $\sqrt{2} t = m(t)+ \frac{3}{2\sqrt{2}}\ln(t)$, by putting  $\overline{x} \defi x+ \frac{3}{2\sqrt{2}}\ln(t)$, 
we reformulate the above as
\beq \label{ugly_three}
\gamma(r)^{-1}  \psi(r, t, \overline{x}  + m(t)) \leq \overline{u}(t, \overline{x} + m(t)) \leq \gamma(r)   \psi(r, t, \overline{x} + m(t)).
\eeq
(The bounds in \eqref{ugly_three} hold for $\overline{x}\geq 8r$). \\

We lighten notations by setting
\beq
G(t,r; \overline{x}, y') \defi \overline{u}(r, y'+\sqrt{2}r) \cdot e^{y'\sqrt{2}} \cdot \exp\left( -\frac{(y'-\overline{x}+\frac{3}{2\sqrt{2}}\ln t)^2}{2(t-r)} \right),
\eeq
and rewrite \eqref{ugly_three} accordingly:
\beq \bea \label{ugly_four}
\psi(r, t, \overline{x}  + m(t)) & = \frac{ t^{3/2} e^{-\overline{x} \sqrt{2}}}{\sqrt{t-r}}   \int_0^\infty \frac{ dy'}{\sqrt{2\pi}} \cdot \left\{ 1- e^{-2 y'\frac{\overline{x}}{t-r} } \right\} \cdot  G(t,r; \overline{x}, y') \\
& = t(1+o(1))  e^{-\overline{x} \sqrt{2}} \int_0^\infty \frac{ dy'}{\sqrt{2\pi}} \cdot \left\{ 1- e^{-2 y'\frac{\overline{x}}{t-r} } \right\} \cdot  G(t,r; \overline{x}, y').
\eea \eeq

By a dominated convergence argument \cite[Prop. 8.3 and its proof]{bramson_monograph} one can prove that 
\beq \label{ugly_five}
C(r) \defi \lim_{t\to \infty} \int_0^\infty  2y'   G(t,r; \overline{x}, y')  \frac{dy'}{\sqrt{2\pi}}, 
\eeq
exists, {\it uniformly} for $\overline{x}$ in compacts. In fact, Bramson's argument easily extends to the case where $\overline{x} = o(\sqrt{t})$ (to see this, 
one simply expands the quadratic term in the Gaussian density appearing in the definition of the function $G$). Moreover, $C(r) \to C$ as $r\to \infty$, 
with $C$ as in \eqref{gumbel_like}, see \cite[p. 145-146]{bramson_monograph}. By Taylor expansion, 
\beq \label{ugly_approx}
2 y'\frac{\overline{x}}{t-r} - \frac{2(y')^2 \overline{x}^2}{(t-r)^2} +\frac{f(t, r; x, y')}{(t-r)^3} \leq\left\{ 1- e^{-2 y'\frac{\overline{x}}{t-r} } \right\}  \leq 2 y'\frac{\overline{x}}{t-r} \,,
\eeq
for some function $f(t, r; x, y')$ which is integrable with respect to $G(t,r; \overline{x}, y')dy'$. \\

Plugging \eqref{ugly_approx} in \eqref{ugly_four} we get the bounds
\beq\bea \label{ugly_six}
\overline{u}(t, \overline{x} + m(t)) \quad \geq & \quad \overline{x} e^{-\overline{x} \sqrt{2} } \int_0^\infty  2y' G(t,r; \overline{x}, y')   \frac{dy'}{\sqrt{2\pi}} + \\
& \hspace{2cm}  +\frac{\overline{x}^2 e^{-\overline{x} \sqrt{2} }}{t-r}   \int_0^\infty  2(y')^2 G(t,r; \overline{x}, y')   \frac{dy'}{\sqrt{2\pi}} +O((t-r)^{-2}), \\ 
\overline{u}(t, \overline{x} + m(t))  \quad \leq & \quad  \overline{x} e^{-\overline{x} \sqrt{2} }   \int_0^\infty  2y' G(t,r; \overline{x}, y')   \frac{dy'}{\sqrt{2\pi}},
\eea \eeq
for large enough $t$. 

The claim of the Lemma then follows by taking $\overline{x} \defi X(t)$ in  \eqref{ugly_six} and using\eqref{ugly_five}. 

\end{proof}

\begin{proof}[Proof of Theorem \ref{convergence_to_der_mart}]
This is a straightforward application of Lemma \ref{fundamental_tail} and the convergence of the derivative martingale. First we write
\beq \bea 
& \PP\left[M(T\cdot s) \in \mathcal D  \mid \mathcal F_{R_T} \right]  = \PP\left[M(T\cdot s) \leq D  \mid \mathcal F_{R_T} \right]- \PP\left[M(T\cdot s) \leq d  \mid \mathcal F_{R_T} \right].
\eea \eeq
We will prove almost sure convergence of the first term, the second being identical. 
Since $s$ is  in $(\vare, 1)$, we have $R_T = o(T\cdot s)$ for $T\uparrow \infty$. Therefore, 
by \eqref{splitting_m} and \eqref{self_sim}, and writing $\PP_M$ for integration with respect to the maximum,
\beq \bea
& \PP\left[ M(T\cdot s) \leq D  \mid \mathcal F_{R_T}  \right] = \\
& \quad = \prod_{k\leq n(R_T)} \PP_M\left[ M(Ts-R_T) \leq D + y_k(R_T) \mid \mathcal F_{R_T}  \right] \\
& \quad = \prod_{k\leq n(R_T)} \left\{ 1- \PP_M\left[ M(Ts-R_T) > D + y_k(R_T) \mid \mathcal F_{R_T}  \right]\right\} \\
& \quad = \exp\left( \sum_{k\leq n(R_T)} \ln\big( 1- \PP_M\left[ M(Ts-R_T) > D + y_k(R_T) \right]\big)\right) \\
\eea \eeq
It immediately follows from the almost sure convergence of the derivative martingale that
\beq \label{y_to_infty}
\lim_{R_T \uparrow \infty} \min_{k\leq n(R_T)} y_k(R_T) = +\infty  \text{ almost surely.}
\eeq
We may therefore use Lemma \ref{fundamental_tail} to establish upper- and lower bounds for the probability 
of the maximum being larger than $D+y_k(R_T)$, precisely:
\beq \bea \label{gamma}
& C \gamma(r)^{-1} \big\{D+y_k(R_T)\big\} \exp\Big\{-\sqrt{2}(D+y_k(R_T)\Big\} \leq \\
& \hspace{2cm} \leq \PP_M\left[ M(Ts-R_T) > D + y_k(R_T) \right] \leq\\ 
& \hspace{3cm} \leq C \gamma(r) \big\{D+y_k(R_T)\big\} \exp\Big\{-\sqrt{2}(D+y_k(R_T)\Big\} \left(
1 + \frac{(D+y_k(R_T))}{Ts-R_T-r} \right),
\eea \eeq
for $Ts-R_T\geq 8 r>0$. 

The main contribution to both bounds above comes from the $z_k$-terms  defined in \eqref{defi_y}. 
Precisely, we write \eqref{gamma} as 
\beq \bea \label{gamma_two}
& C \gamma(r)^{-1}e^{-\sqrt{2} D} z_k(R_T)+\omega_k(R_T) \leq \\
& \hspace{3cm} \leq \PP_M\left[ M(Ts-R_T) > D + y_k(R_T) \right] \leq\\ 
& \hspace{5cm} \leq C \gamma(r)e^{-\sqrt{2} D} z_k(R_T) + \Omega_k(R_T)\ ,
\eea \eeq
where
\beq
\bea
\omega_k(R_T) &\defi C~D~ \gamma(r)^{-1}e^{-\sqrt{2} D} e^{-\sqrt{2} y_k(R_T)}, \\
\Omega_k(R_T) &\defi  C~ D~ \gamma(r) \left(1 + \frac{(D+y_k(R_T))}{Ts-R_T-r} \right) \cdot e^{- \sqrt{2}D}e^{-\sqrt{2} y_k(R_T)}\ .
\eea
\eeq

By \eqref{gamma_two}, using that $-a\leq \ln(1-a)\leq -a+a^2/2$ (valid for $0< a< 1/2$), and with the above notations, 
we obtain
\beq \bea \label{conv_con_max_both}
& \exp\left(- C \gamma(r)^{-1} e^{-\sqrt{2}D} Z(R_T)- \sum_{k\leq n(R_T)} \omega_k(R_T) \right)\\
& \hspace{2cm} \leq \PP\left[ M(T\cdot s) \leq D  \mid \mathcal F_{R_T}  \right] \leq \\
&  \exp\left(- C \gamma(r)  e^{-\sqrt{2}D} Z(R_T) + \frac{C^2}{2} \gamma(r)^2  e^{-2\sqrt{2}D} Z^{(2)}(R_T) + \sum_{k\leq n(R_T)} ( -\Omega_k(R_T) + \frac{\Omega_k(R_T)^2}{2}) \right).
\eea \eeq
where $Z^{(2)}(R_T)\defi \sum_{k \leq n(R_T)} y_k(R_T)^2e^{-2\sqrt{2} y_k(R_T)} $.
To see that the $\omega$ terms in the lower bound do not contribute in the limit $T\uparrow \infty$ (recall that $R_T \uparrow \infty$ as well), 
we observe that for some  $\kappa>0$ large enough and $Y(R_T)$ as in \eqref{defi_martingale}, 
\beq
 \sum_{k\leq n(R_T)} \omega_k(R_T)  \leq \kappa \cdot Y(R_T) \to 0 \text{ almost surely,}
\eeq
by \eqref{y_to_zero}. Therefore the $\omega$ term in the lower bound do not contribute in the limit $T\uparrow \infty$.

Concerning the upper bound, the same argument as for the $\omega$ term together with the fact that $Z(R_T)\to Z$ as $T\to\infty$ by \eqref{y_to_zero} imply
that
\beq
 \sum_{k\leq n(R_T)} \Omega_k(R_T)  \to 0 \text{ almost surely.}
\eeq
The same is thus also true for  $\sum_{k\leq n(R_T)} \Omega_k(R_T)^2$. It remains to show that $Z^{(2)}(R_T)\to 0$ almost surely, but this is evident
since this sum is bounded from above by
\beq 
\max_{k\leq n(R_T)}\left( y_k(R_T)^2 e^{-\sqrt{2}y_k(R_T)}\right) \times Y(R_T),
\eeq
and both terms tend to zero,  a.s., 
as $T\uparrow\infty$   by \eqref{y_to_infty} and  \eqref{y_to_zero}. 
Therefore,
by \eqref{conv_con_max_both},
\beq
\lim_{T\uparrow \infty} \PP\left[ M(T\cdot s) \leq D  \mid \mathcal F_{R_T}  \right] = \exp\left(-C Z e^{-\sqrt{2} D} \right) \text{ almost surely.}
\eeq
This concludes the proof of Theorem \ref{convergence_to_der_mart}. 
\end{proof}

\section{The strong law of large numbers} \label{intricate}
This section is organized as follows.
We introduce in subsection \ref{path_loc_sec} a procedure concerning properties of the {\it paths} of extremal particles which we will refer to as {\it localization}. 
It is based on the description of the genealogies of extremal particles established in \cite{abk_genealogies}.
The details of the proof are given in subsection \ref{implementing the strategy}.

\subsection{Preliminaries and localization of the paths} \label{path_loc_sec}

The following fundamental result by Bramson provides bounds to the right tail of the maximal displacement. 
These bounds are not optimal (they are surpassed by those of Lemma \ref{fundamental_tail}, which are tight), but they are sufficient and simpler.

\begin{lem}\cite[Section 5]{bramson} \label{easy_max_bramson} Consider a branching Brownian motion $\{x_j(t)\}_{j\leq n(t)}$.  Then, for $0\leq y \leq t^{1/2}$ and $t\geq 2$,
\beq
\PP\left[ \max_{j\leq n(t)} x_j(t)-m(t)\geq y \right] \leq \gamma (y+1)^2 e^{-\sqrt{2}y},
\eeq
where $\gamma$ is independent of $t$ and $y$. 
\end{lem}

We also recall an important property of the paths of extremal particles established by the authors in 
\cite{abk_genealogies}. We 
introduce some notation. With  $t \in \R_+$ and $\gamma>0$, we define  
\beq \label{defi_f}
f_{\gamma,t}(s) \defi \begin{cases}
                    s^\gamma & 0\leq s \leq t/2, \\
		    (t-s)^\gamma & t/2\leq s \leq t. 
                   \end{cases}
\eeq
We now choose values 
\beq 
0 <\alpha<  1/2 < \beta < 1,
\eeq 
and introduce the  {\it time-$t$ entropic} envelope, and the {\it time-$t$ lower} envelope respectively: 
\beq
F_{\alpha,t} (s) \defi \frac{s}{t} m(t) - f_{\alpha,t}(s), \quad 0\leq s\leq t,
\eeq
and
\beq
F_{\beta,t} (s) \defi \frac{s}{t} m(t) - f_{\beta,t}(s), \quad 0\leq s\leq t.
\eeq
($m(t)$ is the level of the maximum of a BBM of length $t$). By definition, 
\beq 
F_{\beta,t} (s) < F_{\alpha,t} (s),
\eeq
and
\beq 
F_{\beta,t} (0) = F_{\alpha,t}(0) =0, \quad F_{\beta,t} (t) = F_{\alpha,t}(t) =m(t). 
\eeq
The space/time region between the entropic and lower envelopes will be denoted throughout as the  {\it time-t tube}, or simply the {\it  tube}. 

By a slight abuse of notation, given a particle $k\leq n(t)$ which is at position $x_k(t)$ at time $t$, we refer to its {\it path} as $x_k(s)$ where $0\leq s\leq t$. 
Moreover, we will say that a particle $k$ is {\it localized} in the time $t$-tube during the interval $(r, t-r)$ if and only if
$$
F_{\beta,t} (s)\leq x_k(s) \leq F_{\alpha,t}(s), \forall s\in (r, t-r) \ .
$$
We  say that it is {\it not localized} if the above requirement fails  
for some 
$s$ in $(r,t-r)$.
The following proposition gives strong bounds to the probability of finding
 particles that are close to the level of the maximum at given times but not
 localized. 
It follows directly from the bounds derived in the course of the proof
of \cite[Corollary 2.6]{abk_genealogies}, cf. equations (5.5), (5.54), (5.62) and (5.63).

\begin{prop} \label{path_loc_prop}
Let the subset $\mathcal D = [d, D]$ be given,  with $-\infty < d < D \leq \infty$. There exist $r_o, \delta>0$ depending on 
$\alpha, \beta$ and $\mathcal D$ such that
\beq \label{tube} \bea
& \sup_{t\geq 3r_o} \PP\big[ \exists_{k\leq n(t)} \, \overline{x_k(t)} \in \mathcal D \; \text{but the path is not} \\
& \hspace{3cm} \text{localized in the time-$t$ tube during} \; (r, t- r)\big] \leq \exp\left(-r^\delta\right).
\eea \eeq
\end{prop}

What lies behind the Proposition is a phenomenon of "energy vs. entropy" which is absolutely fundamental 
for the whole picture. This is explained in detail in \cite{abk_genealogies}, but, for the reader's convenience, we 
briefly sketch the argument. 

As it turns out, at any given time $s\in (r,t-r)$ well 
inside the lifespan of a BBM, there are simply not enough particles lying above the 
entropic envelope for their offspring to make the jumps which eventually bring them to the edge at time
$t$. On the other hand, although there are plenty of ancestors lying below the lower envelope, 
their position is so low that again none of their offspring will make it to the edge at time $t$. A delicate balance  between number and positions of ancestors 
has to be met, and this feature is fully captured by the tubes. \\

With $\delta= \delta(\alpha, \beta, \mathcal D)$ as in Proposition \ref{path_loc_prop} we define 
\beq \label{choice_r_T}
r_T \defi (20 \ln T)^{1/\delta} \ .
\eeq 
We now consider the maximum of the particles at time $s$ that are also localized during the interval $(r_T, s-r_T)$, see Figure \ref{all_loc_max} for a graphical representation.  
We denote this maximum by $M_{\text{loc}}(s)$. 
\begin{figure} 
\includegraphics[scale=0.4]{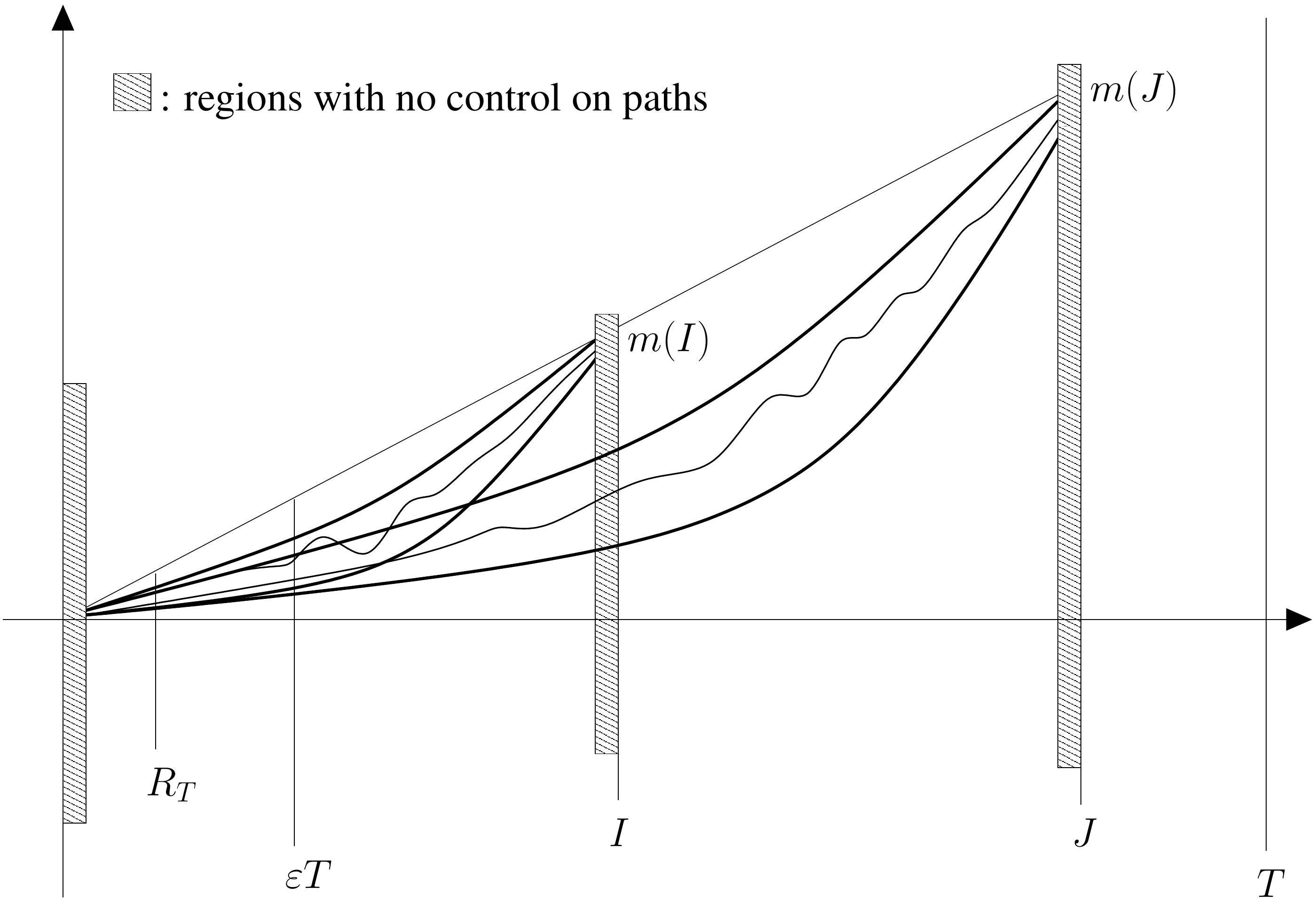} 
\caption{Maxima at different times $I,J$ are localized.}
\label{all_loc_max}
\end{figure} 
With this notation, by Proposition \ref{path_loc_prop} and the choice \eqref{choice_r_T}, 
\beq \bea \label{bc}
0\leq \PP\left[M(s) \in \mathcal D \right] - \PP\left[ M_{\text{loc}}(s) \in \mathcal D\right] \leq \frac{1}{T^{20}}.
\eea \eeq
We pick $R_T \defi 40\cdot r_T$, with $r_T$ as in \eqref{choice_r_T}. This choice clearly satisfies $R_T = o(\sqrt{T})$ as required in Theorem
\ref{convergence_to_der_mart}. We emphasize that the prefactor is a choice. Only the condition $R_T > r_T$ is needed. \\

We assume henceforth without loss of generality that both $T$ and $\vare T$ are integers. 

\subsection{Implementing the strategy} \label{implementing the strategy}
Recall that Theorem \ref{slln} asserts that 
\beq \label{claim_rest}
\text{Rest}_{\vare, \mathcal D}(T) \defi  \frac{1}{T}  \int_{\vare T}^T \Big(\1\{M(s) \in \mathcal D  \} - \PP\left[M(s) \in \mathcal D  \mid \mathcal F_{R_T} \right] \Big) ds
\eeq
tends to zero as $T$ goes to $\infty$.
In order to prove the claim, we consider $\text{Rest}^{\text{loc}}_{\vare, \mathcal D}(T)$, defined as $\text{Rest}_{\vare, \mathcal D}(T)$ but with the requirement that 
all particles in $\mathcal D$ are localized: 
\beq \bea \label{important_loc} 
\text{Rest}^{\text{loc}}_{\vare, \mathcal D}(T) \defi \frac{1}{T}  \int_{\vare T}^T \Big(\1\{M_{\text{loc}}(s) \in \mathcal D  \} - \PP\left[M_{\text{loc}}(s) \in \mathcal D  \mid \mathcal F_{R_T} \right] \Big) ds \ .
\eea \eeq

We now claim that the large $T$-limit of $\text{Rest}^{\text{loc}}_{\vare, \mathcal D}(T)$ and that of $\text{Rest}_{\vare, \mathcal D}(T)$ coincide  (provided one of the two exists, but this will become 
apparent below). 

\begin{lem} \label{if_exist_coincide} With the above notation, 
\beq
\lim_{T\uparrow \infty} \Big(\emph{Rest}_{\vare, \mathcal D}(T)-\emph{Rest}^{\emph{loc}}_{\vare, \mathcal D}(T)\Big) = 0 \text{ almost surely. }
\eeq
\end{lem}
%

\begin{proof}[Proof of Lemma \ref{if_exist_coincide}]
We have
\beq \bea
&  \text{Rest}_{\vare, \mathcal D}(T) - \text{Rest}^{\text{loc}}_{\vare, \mathcal D}(T)= \\
&  = \frac{1}{T}  \int_{\vare T}^T \Big(\1\{M(s) \in \mathcal D  \} -\1\{M_{\text{loc}}(s) \in \mathcal D  \}\Big) ds \\
& \qquad - \frac{1}{T}  \int_{\vare T}^T \Big(\PP\left[ M(s) \in \mathcal D  \mid \mathcal F_{R_T} \right]-\PP\left[  M_{\text{loc}}(s) \in \mathcal D  \mid \mathcal F_{R_T} \right]\Big) ds \\
& \equiv \boldsymbol{(1)_{T, \vare}} - \boldsymbol{(2)_{T, \vare}}.
\eea \eeq
The proof that $\lim_{T\uparrow \infty} \boldsymbol{(1)_{T, \vare}} =  0$ and $\lim_{T\uparrow \infty} \boldsymbol{(2)_{T, \vare}} = 0$ (almost surely) is identical and
relies on an application of the Borel-Cantelli lemma. We thus prove only the first limit. Let $\epsilon >0$. By the Chebeychev inequality, 
\beq \bea 
&\PP\left[   \boldsymbol{(1)_{T, \vare}} >\epsilon \right]  \leq \frac{1}{T \epsilon} \int_{\vare T}^T \Big(
\PP\left[ M(s) \in \mathcal D  \right]-\PP\left[  M_{\text{loc}}(s) \in \mathcal D  \right]\Big) ds \leq \\
& \stackrel{\eqref{bc}}{\leq} \frac{1-\vare}{\epsilon} T^{-20},
\eea\eeq
which is summable in $T$ (recalling that we assume $T \in \N$). Therefore, by Borel-Cantelli, 
\beq 
\PP\left[ \{\boldsymbol{(1)_{T, \vare}} > \epsilon\} \; \text{infinitely often} \right] = 0.
\eeq
As the above holds for all $\epsilon>0$ we have that $\boldsymbol{(1)_{T, \vare}}$ converges to
$0$ as $T\uparrow \infty$ almost surely, and concludes the proof of Lemma \ref{if_exist_coincide}.
\end{proof}

The following result is the major tool to establish the SLLN for the term $\text{Rest}^{\text{loc}}_{\vare, \mathcal D}(T)$.
(By Lemma \ref{if_exist_coincide}, this will then imply that the same is true for $\text{Rest}_{\vare, \mathcal D}(T)$).
The result is a small extension of a theorem of Lyons \cite[Theorem 1]{lyons},
 where the statement is given for the sum of random variables.
\begin{teor} \label{lyons_int}
Consider a process $\{X_s\}_{s\in \R_+}$ such that $\E[X_s]=0$ for all $s$. Assume furthermore that the random variables are 
uniformly bounded, say $\sup_s |X_s|\leq 2$ almost surely. If
\beq \label{sum_corr}
\sum_{T=1}^\infty \frac{1}{T} \E\Big[\Big| \frac{1}{T} \int_0^T X_s ~ds \Big|^2\Big]<\infty, 
\eeq
then
\beq
\lim_{T\to\infty} \frac{1}{T} \int_0^T X_s ~ ds = 0, \text{ almost surely.}
\eeq
\end{teor}

\begin{proof} 
The extension to integrals is straightforward. In fact, by the summability assumption, 
we can find a subsequence $T_k\in\N$ of times such that
\beq
\sum_{k=1}^\infty  \E\Big[\Big| \frac{1}{T_k} \int_0^{T_k} X_t ~dt \Big|^2\Big]<\infty
\eeq
where $T_k\to\infty$ and $T_{k+1}/T_k\to 1$. (See \cite[Lemma 2]{lyons}). Therefore by Fubini, the sum without the expectation is almost surely finite, 
and we must have
\beq
\lim_{k\to\infty}\frac{1}{T_k} \int_0^{T_k} X_t ~dt\to 0 \text{ almost surely .}
\eeq
It remains to show this is true for all $T\in\N$. This is easy since the variables are bounded. For any $T$, there exists $k$
such that $T_k\leq T \leq T_{k+1}$. Thus
\beq
\Big| \frac{1}{T} \int_0^T X_t ~dt \Big| \leq \Big| \frac{1}{T_k} \int_0^{T_k} X_t ~dt \Big|  + \max_{1\leq s \leq T_{k+1}-T_k}\Big| \frac{1}{T_k} \int_{T_k}^{T_{k}+s} X_t ~dt \Big| \ .
\eeq
The first term goes to zero by the previous argument. The second term goes to zero since
\beq
\max_{1\leq s \leq T_{k+1}-T_k}\Big|\frac{1}{T_k}\int_{T_k}^{T_{k}+s} X_t ~dt \Big|\leq \frac{T_{K+1}-T_k}{T_k} \ ,
\eeq
and $T_{k+1}/T_k \to 1$.
\end{proof}

Note that
\beq \bea
\text{Rest}^{\text{loc}}_{\vare, \mathcal D}(T) & =  \frac{1}{T} \int_{\vare T}^T \Big(\1\{M_{\text{loc}}(s) \leq D\} - \PP\left[M_{\text{loc}}(s) \leq D  \mid \mathcal F_{R_T} \right]\Big) ds \\
&  \qquad  - \frac{1}{T} \int_{\vare T}^T \Big(\1\{M_{\text{loc}}(s) \leq d\} - \PP\left[M_{\text{loc}}(s) \leq d  \mid \mathcal F_{R_T} \right]\Big) ds \\
& \defi \frac{1}{T} \int_{\vare T}^T X_s^{\{D\}}ds - \frac{1}{T} \int_{\vare T}^T X_s^{\{d\}} ds,
\eea \eeq
with obvious notations. The goal is thus to prove that both integrals satisfy the assumptions of Theorem \ref{lyons_int}. 
We address the first integral, the proof for the second being identical.
By  construction, $\big| X_s^{\{D\}} \big| \leq 2$ a.s. for all $s$, and   
\beq
\E\left[ X_s^{\{D\}} \right] =0\ .
\eeq
It therefore suffices to check the assumption concerning the summability of correlations. Let 
\beq \label{correlators}
\widehat C_T(s, s') \defi \E\left[  X_s^{(D)} \cdot X_{s'}^{(D)}  \right], 
\eeq
Note that by the properties of conditional expectation
\beq \bea
\widehat C_T(s, s') & = \E\Bigg[\Big(\1\{M_{\text{loc}}(s) \leq D\} - \PP\left[M_{\text{loc}}(s) \leq D  \mid \mathcal F_{R_T} \right]\Big) \times \\
& \hspace{2cm} \times \Big(\1\{M_{\text{loc}}(s')\leq D\} - \PP\left[M_{\text{loc}}(s') \leq D  \mid \mathcal F_{R_T} \right]\Big) \Bigg] \\
& = \E\Bigg[\Bigg( \PP\left[M_{\text{loc}}(s)\leq D, M_{\text{loc}}(s') \leq D \mid \mathcal F_{R_T} \right]\\
& \hspace{2cm}- \PP\left[M_{\text{loc}}(s) \leq D  \mid \mathcal F_{R_T} \right] \times \PP\left[M_{\text{loc}}(s') \leq D  \mid \mathcal F_{R_T} \right]\Bigg) \Bigg].
\eea \eeq
We claim that
\beq \bea  \label{summability_lyons_two}
& \sum_T \frac{1}{T} \E\left[\Big| \frac{1}{T} \int_{\vare T}^T X_s^{(D)} ~ds \Big|^2\right]  =  2 \sum_T \frac{1}{T^3} \int_{\vare T}^T ds  \int_{s}^T ds' \widehat C_T(s, s') \text{ is finite.} 
\eea \eeq
In order to see this, and proceeding with the program outlined at the end of Section \ref{conjecture}, we now specify the concept of times {\it  well separated} from each other. 
Choose $0< \xi < 1$ and split the integration according to the distance between $s$ and $s'$:
\beq \bea \label{summability_lyons_three}
&  \frac{1}{T^3}\int_{\vare T}^T ds  \int_{s}^T ds' (\cdot) =  \frac{1}{T^3}\int_{\vare T}^T ds  \int_{s}^{s+T^\xi} ds' (\cdot) +  \frac{1}{T^3} \int_{\vare T}^T ds  \int_{s+T^\xi}^T ds' (\cdot) .
\eea \eeq 
The contribution of the first term on the r.h.s. above is negligible due to the uniform boundedness of the integrand  and to the choice $0<\xi<1$.
We are thus left to prove that the contribution to \eqref{summability_lyons_two} of the second term in \eqref{summability_lyons_three} is finite. 
The following is the key estimate. 

\begin{teor} \label{uniform_bound_lyons}
There exists a finite $T_o$ such that the following holds for $T\geq T_o$:
for some $\epsilon_1, \epsilon_2 > 0$ \emph{not} depending on $T$ (but on the other underlying parameters), the bound
\beq \label{uniform_c_lyons_zero}
\widehat C_T(s, s') \leq  (\ln  T)^{\epsilon_1} e^{-(\ln T)^{\epsilon_2}}
\eeq
holds uniformly for all $s, s'$ such that $\vare T \leq s < s' \leq T$ and $s'-s>T^\xi$. 
\end{teor}
The estimate directly implies the desired summability of the second term in \eqref{summability_lyons_three}. 
This concludes the proof Theorem \ref{slln}. The proof of the estimate is somewhat lengthy and done in the next section.

\section{Uniform bounds for the correlations.} \label{technicalities}
We use here $I$ and $J$ to denote the two times $s, s'$ from the statement of Theorem \ref{uniform_bound_lyons}. 

$\widehat C_T(I,J)$ is the expectation of the random variable
\beq \bea \label{start}
& \hat c_T(I, J)\defi \PP\left[M_{\text{loc}}(I)\leq D, M_{\text{loc}}(J) \leq D \mid \mathcal F_{R_T} \right]\\
& \hspace{4cm}- \PP\left[M_{\text{loc}}(I) \leq D  \mid \mathcal F_{R_T} \right] \times \PP\left[M_{\text{loc}}(J) \leq D  \mid \mathcal F_{R_T} \right].
\eea
\eeq
We rewrite these conditional probabilities using the Markov property of BBM, considering independent BBM's starting at their respective position at time $R_T$
and shifting the time by $R_T$. This requires some additional notation.
Take
$$
I_T \defi I-R_T, J_T \defi J-R_T \ ,
$$
and note that $m(I) = m(I_T)+ \sqrt{2}I_T +o(1)$ as $T\uparrow \infty$.  
We consider the collection $\{y_k(R_T)\defi \sqrt{2}R_T-x_k(R_T)\}_{k\leq n(R_T)}$ where the $\{x_k(R_T)\}$ are the position of the particles of the original BBM at time $R_T$.  
Let $\{\tilde x_l(J_T), l\leq n(J_T)\}$ be a BBM starting at zero, of length $J_T$, and  of law $\tilde \PP$ independent of $\PP$.
We write $\tilde M_{\text{loc}}(J_T)$ for the maximum shifted by $m(J_T)$ of this collection, restricted to $l$'s satisfying 
\beq \bea \label{cond_path_after_R_T}
&y_k(R_T)+\frac{s'}{J}m(J)-f_{\beta, J}(R_T+s') \leq  \tilde x_l(s') \leq y_k(R_T)+\frac{s'}{J}m(J)-f_{\alpha, J}(R_T+s') \ ,
\eea \eeq
for $0\leq s' \leq J_T-r_T$ (the "shifted" $J$-tube). 
Similarly, $\tilde M_{\text{loc}}(I_T)$ is the maximum shifted by $m(I_T)$ of the positions of the particles at time $I_T$ 
with the localization condition 
\beq \bea \label{cond_path_after_R_T_two}
&y_k(R_T)+\frac{s'}{I}m(I)-f_{\beta, I}(R_T+s') \leq  \tilde x_l(s') \leq y_k(R_T)+\frac{s'}{J}m(J)-f_{\alpha, J}(R_T+s'),
\eea \eeq
for $0\leq s' \leq I_T-r_T$ (the "shifted" $I$-tube). 
Note that the localization depends on $k$ (in fact on $y_k(R_T)$). We drop this dependence in the notation $\tilde M_{\text{loc}}$ for simplicity.

By the Markov property, the first conditional probability in $\hat c_T(I, J)$ can be written in terms of the shifted process just defined:
\beq \bea \label{c_one}
& \PP\left[ M_{\text{loc}}(I)\leq D, M_{\text{loc}}(J) \leq D \mid \mathcal F_{R_T} \right]\\
& = \prod^\star_{k\leq n(R_T)} \tilde \PP \left[\tilde M_{\text{loc}}(I_T) \leq   D +y_k(R_T), \tilde M_{\text{loc}}(J_T) \leq   D+y_k(R_T) \right],
\eea \eeq
where the product runs over all the particles $k$'s at time $R_T$ whose path is localized in the intersection of the $I-$ and $J-$tubes during the interval $(r_T, R_T)$. 
The restriction to localized positions at time $R_T$ is weaker and sufficient for our purpose:
\beq \label{loc at time R_T}
k=1\dots n(R_T)\, \text{such that}\; y_k(R_T) \in \left(R_T^\alpha + \Omega_T, R_T^\beta+\Omega_T\right) \hspace{1cm} (\bigtriangleup).
\eeq  
(Here and henceforth, we will use $\Omega_T$ to denote a negligible term, which is not necessarily the 
same at different occurences. In the above case it holds $\Omega_T = O(\ln \ln T)$ by definition of the tubes). 
We thus get that \eqref{c_one} is at most
\beq \bea \label{c_four}
& \prod_\bigtriangleup \tilde \PP \left[\tilde M_{\text{loc}}(I_T) \leq   D +y_k(R_T), \tilde M_{\text{loc}}(J_T) \leq   D+y_k(R_T) \right]\ .
\eea\eeq

Let 
\beq \label{p}
\wp(I_T; y_k(R_T))  \defi \tilde \PP \left[\tilde M_{\text{loc}}(I_T)>  D +y_k(R_T) \right],
\eeq
(analogously for $J_T$) and
\beq \label{p_due}
\wp(I_T, J_T; y_k(R_T) ) \defi \tilde \PP\left[\tilde M_{\text{loc}}(I_N)>  D +y_k(R_T)\; \text{and}\; \tilde M_{\text{loc}}(J_T) >  D+y_k(R_T) \right].
\eeq
Finally, define 
\begin{eqnarray}
 \widehat Z(\cdot; R_T) & \defi &\sum_\bigtriangleup \wp(\cdot; y_k(R_T)),\\
\mathcal R_T & \defi& \frac{1}{2} \sum_\bigtriangleup \Big\{\wp(I_T; y_k(R_T))+\wp(J_T; y_k(R_T)) - \wp(I_T,J_T; y_k(R_T)\Big\}^2. \label{rest_term_quadratic}
\end{eqnarray}

\begin{prop} \label{a_priori_ref}
With the above definitions, 
\beq \label{upper_bastaa}
0\leq \hat c_{T}(I,J) \leq  \widehat Z(I_T, J_T; R_T ) + \mathcal R_T,
\eeq
almost surely, for $T$ large enough.
\end{prop}

\begin{proof}
In the notation introduced above, one has
\beq \bea 
\label{log_P}
& \prod_\bigtriangleup \tilde \PP \left[\tilde M_{\text{loc}}(I_T) \leq   D +y_k(R_T), \tilde M_{\text{loc}}(J_T) \leq   D+y_k(R_T) \right]\\
& \qquad = \exp\left\{\sum_\bigtriangleup \ln\Big[ 1-\wp(I_T; y_k(R_T))- \wp(J_T; y_k(R_T)) + \wp(I_T; J_T; y_k(R_T))  \Big] \right\}\ .
\eea\eeq
Note that for all $k \in \bigtriangleup$
\beq \label{bram_simple_bound}
\wp(I_T; y_k(R_T)) \leq \tilde \PP\left[ \tilde M(I_T) \geq D +y_k(R_T) \right] \leq \gamma(1+y_k(R_T)+D)^2 e^{-\sqrt{2}(y_k(R_T)+D)} \ .
\eeq
The first inequality holds by dropping the localization condition. Therefore, this
can be made arbitrarily small  (uniformly in $k$) by choosing $T$ large enough. 
The same obviously holds for 
$\wp(J_T; y_k(R_T))$ and $\wp(I_T, J_T; y_k(R_T))$. 
Choose $T$ large enough so that
\beq 
\sup_\bigtriangleup \max\{\wp(I_T; y_k(R_T)), \wp(J_T; y_k(R_T)), \wp(I_T, J_T; y_k(R_T))\} \leq {1/6}.
\eeq
Coming back to \eqref{log_P}, using that
\beq \label{log-approx}
-a\leq \ln(1-a) \leq -a +a^2/2 \quad (0\leq a\leq 1/2),
\eeq
(with $a\defi \wp(I_T; y_k(R_T))+\wp(J_T; y_k(R_T)) - \wp(I_T; J_T; y_k(R_T))$, for $k\in \bigtriangleup$), we  get that \eqref{log_P} is at most 
\beq \label{c_five} \bea
& \exp \Bigg( - \widehat Z(I_T; R_T) - \widehat Z(J_T; R_T) + \widehat Z(I_T,J_T; R_T ) + \mathcal R_T \Bigg).
\eea \eeq
This is an upper bound for the first conditional probability in the definition of $\hat c_{T}(I,J)$.
A similar reasoning, using this time the first inequality in \eqref{log-approx}, yields a lower bound for the second term in $\hat c_{T}(I,J)$, i,e, the product of the conditional probabilities. 
The upshot is:
\beq \label{c_seven}
\hat c_T(I,J) \leq e^{-  \widehat Z(I_T; R_T)-  \widehat Z(J_T; R_T) }\left\{e^{  \widehat Z(I_T, J_T; R_T ) +\mathcal R_T} -1 \right\},
\eeq
almost surely and for large enough $T$. We now use that $e^a-1\leq a\cdot e^a$ (which holds for $a>0$) for the term in the brackets to get that \eqref{c_seven} is at most
\beq \bea \label{c_eight}
& e^{-  \widehat Z(I_T; R_T)-\widehat Z(J_T; R_T )} \left( \widehat Z(I_T, J_T; R_T) +\mathcal R_T \right) \cdot e^{  \widehat Z(I_T, J_T; R_T ) +\mathcal R_T}. \\
\eea \eeq
By construction, $\widehat Z(I_T, J_T; R_T)\leq \min\left\{\widehat Z(I_T; R_T); \widehat Z(I_T; R_T)\right\}$, implying that
\beq \label{one_half}
\widehat Z(I_T, J_T; R_T) -\frac{1}{2} \widehat Z(I_T; R_T) -\frac{1}{2}\widehat Z(J_T; R_T) \leq 0,
\eeq
and therefore \eqref{c_eight} is at most 
\beq \label{c_nine} \bea
\left( \widehat Z(I_T, J_T; R_T) +\mathcal R_T \right) \cdot e^{\mathcal R_T- \frac{1}{2} \widehat Z(I_T; R_T) -\frac{1}{2}\widehat Z(J_T; R_T)}.
\eea \eeq
This is not far from the claim of Proposition \ref{a_priori_ref}. It remains to get rid of the exponential 
on the r.h.s. above. 
Using the bound \eqref{rest_term_quad}, together with the definition of the $\widehat Z$ and rearranging, we 
arrive at
\beq \bea \label{getting_rid}
& \mathcal R_T -  \frac{1}{2} \widehat Z(I_T; R_T) -\frac{1}{2}\widehat Z(J_T; R_T) \\
& \quad \leq  \sum_\bigtriangleup \wp(I_T; y_k(R_T))\left(3 \wp(I_T; y_k(R_T)) -\frac{1}{2}\right) \\
& \hspace{4cm} + \sum_\bigtriangleup \wp(J_T; y_k(R_T))\left(3 \wp(J_T; y_k(R_T)) -\frac{1}{2}\right).
\eea \eeq
In view of \eqref{bram_simple_bound}, we may find $T$ large enough such that the following holds uniformly
for all $k\in \bigtriangleup$:
\beq
 3 \wp(I_T; y_k(R_T)) -\frac{1}{2} \leq 0, \qquad 3 \wp(J_T; y_k(R_T)) -\frac{1}{2} \leq 0, 
\eeq
in which case  {\it all} terms appearing in \eqref{getting_rid} become negative, and this implies that
 \beq \label{c_ten}
\hat c_{T}(I,J) \leq  \widehat Z(I_T, J_T; R_T ) + \mathcal R_T,
\eeq
concluding the proof of Proposition \ref{a_priori_ref}.
\end{proof}

\subsection{Proof of Theorem \ref{uniform_bound_lyons}}
We first observe that the expectation of $\mathcal R_T$ appearing in Proposition \ref{a_priori_ref} gives the right bound in Theorem \ref{uniform_bound_lyons}.
Indeed, using that $(a+b+c)^2\leq 4a^2+4b^2+4c^2$, we get the upper bound  
\beq \bea \label{rest_term_quadratic_three}
\mathcal R_T & = \frac{1}{2} \sum_\bigtriangleup \Big\{\wp(I_T; y_k(R_T))+\wp(J_T; y_k(R_T)) - \wp(I_T,J_T; y_k(R_T)\Big\}^2 \\
& \leq 2 \sum_\bigtriangleup \wp(I_T; y_k(R_T))^2+\wp(J_T; y_k(R_T))^2 + \wp(I_T, J_T; y_k(R_T))^2.
\eea \eeq
Moreover,  
\beq\bea
\wp(I_T, J_T; y_k(R_T))
& \leq \frac{1}{2} \wp(I_T; y_k(R_T))+ \frac{1}{2} \wp(J_T; y_k(R_T)).
\eea \eeq
Inserting this in \eqref{rest_term_quadratic_three}, we get
\beq \bea \label{rest_term_quad}
\mathcal R_T \leq  \sum_\bigtriangleup\left\{ 3 \wp(I_T; y_k(R_T))^2+3 \wp(J_T; y_k(R_T))^2\right\}.
\eea \eeq

By \eqref{rest_term_quad}, \eqref{bram_simple_bound} and \eqref{loc at time R_T}, and for some irrelevant numerical constants $\kappa$,
\beq \bea \label{getting_rid_r}
\E\left[ \mathcal R_T\right] & \leq \kappa \E\left[ \sum_\bigtriangleup  y_k(R_T)^2 e^{-2 \sqrt{2} y_k(R_T)} \right]\\
& \leq \kappa e^{R_T} \int_{R_T^\alpha+\Omega_T}^{R_T^\beta+\Omega_T} y^2 e^{-2\sqrt{2} y} e^{-\frac{(y-\sqrt{2}R_T)^2}{2R_T}} \frac{dy}{\sqrt{2\pi R_T}}  \\
& \leq \kappa R_T^2 e^{-\sqrt{2} R_T^\alpha} = \kappa (\ln T)^{2/\delta} e^{-\kappa (\ln T)^{\alpha/\delta}},
\eea \eeq
which is at most 
\beq \label{bounding_stupid_one}
 \E\left[\mathcal R_T \right] \leq (\ln T)^{\epsilon^{(1)}} e^{-(\ln T)^{\epsilon^{(2)}}}, 
\eeq
for some $\epsilon^{(1)}, \epsilon^{(2)} >0$. 
It will remain to prove that  $\E[\widehat Z(I_T, J_T; R_T )]$ behaves similarly to establish Theorem \ref{uniform_bound_lyons}.

Recall that
\beq
\widehat Z(I_T, J_T; R_T ) =  \sum_\bigtriangleup \wp(I_T, J_T; y_k(R_T) ),
\eeq
and
\beq \label{what p}
\wp(I_T, J_T;y_k(R_T) ) =  \tilde\PP\left[\tilde M_{\text{loc}}(I_T)>  D +y_k(R_T)\; \text{and}\; \tilde M_{\text{loc}}(J_T) >  D+y_k(R_T)\right]. 
\eeq
By definition, \eqref{what p} is the probability to find a particle of the BBM which has two 
extremal descendants, particle $\boldsymbol{(1)}$ say, whose position is above $m(I_T)+D+y_k(R_T)$ at time $I_T$, and particle $\boldsymbol{(2)}$, which lies above $m(J_T)+D+y_k(R_T)$ at time $J_T$.
These two particles also satisfy localization conditions on their paths. 
In other words, this is the probability that the same
ancestor $k$, with (relative) position $y_k(R_T)$, produces children $\boldsymbol{(1)}$ and $\boldsymbol{(2)}$ which are extremal at time $I$ and $J$. 
As these generations are {\it well separated} in time, that is $J-I> T^\xi$ (and thus also $J_T-I_T> T^\xi$), we may expect this probability to be very small. 

In order to see that this is indeed the case, split the probabilities according to whether the {\it most recent} common ancestor of particles $\boldsymbol{(1)}$ and $\boldsymbol{(2)}$ has branched {\it before} time $I_T-r_T$
(with $r_T$ as in \eqref{choice_r_T}), or {\it after}. We write this as
\[ \bea
\wp(I_T, J_T; y_k(R_T) ) &=   \wp(I_T, J_T; y_k(R_T) ; \text{split before}\, I_T-r_T) \\
& \hspace{2cm} + \wp(I_T, J_T; y_k(R_T) ; \text{split after}\, I_T-r_T).
\eea \]
(Figure \ref{tube_fig} illustrates the first case).
\begin{figure}
\includegraphics[scale=0.4]{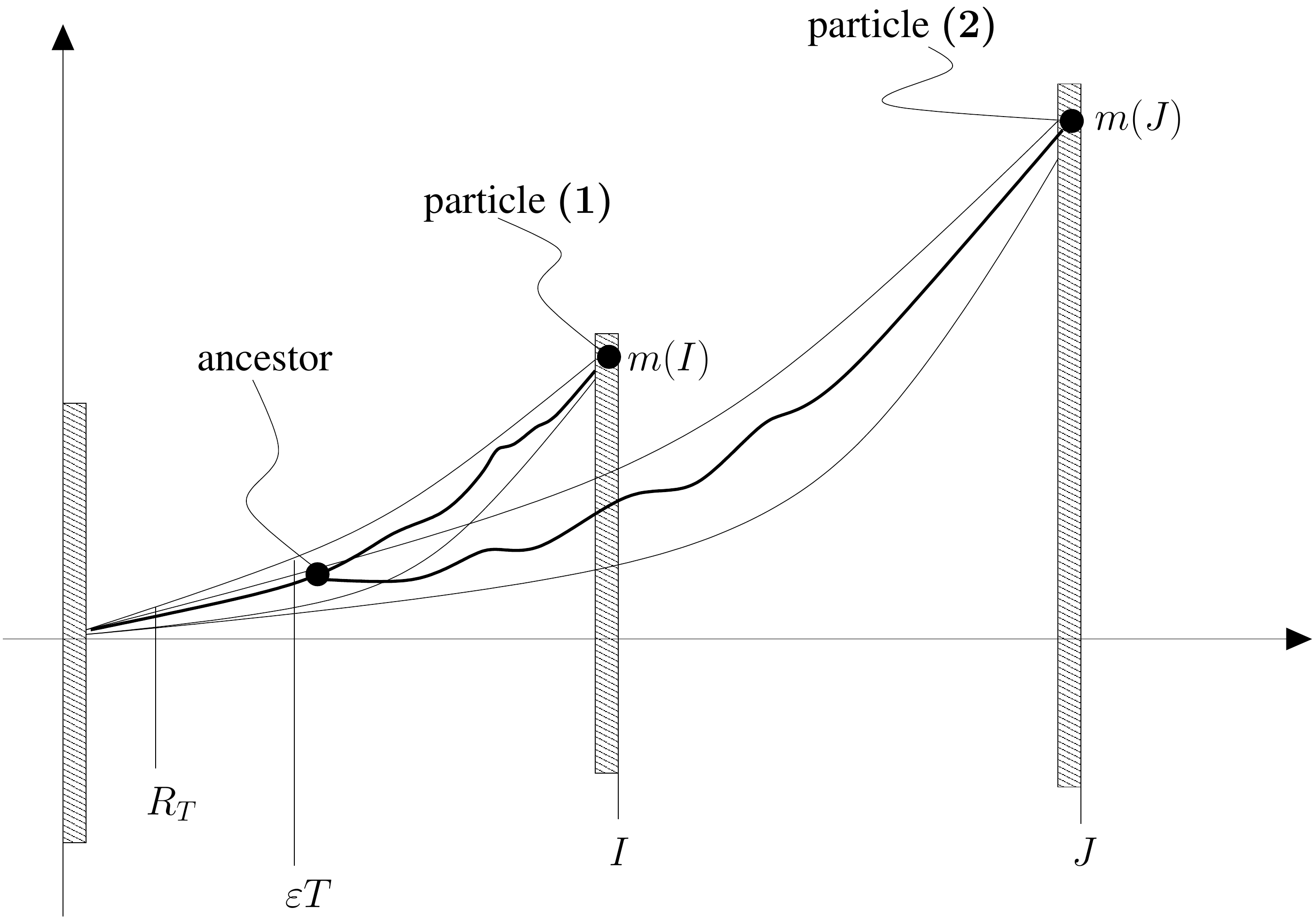}
\caption{Time of branching before $I$}
\label{tube_fig}
\end{figure}
The second probability is in fact zero. Indeed, the condition \eqref{cond_path_after_R_T} implies that the ancestor of $\boldsymbol{(2)}$ at time $I-r_T$ lies
at heights which are at most the level of the entropic envelope associated with $J$.  Since $J-I> T^\xi$ and this is easily seen to be way lower than the {\it lower} envelope of particle $\boldsymbol{(1)}$ associated with time $I$. In other words, the localization tubes of particles $\boldsymbol{(1)}$ and $\boldsymbol{(2)}$ are disjoint if their ancestor split after $I - r_T$.
Hence, the splitting of the ancestor of particles $\boldsymbol{(1)}$ and $\boldsymbol{(2)}$ can
only happen before time $I_T-r_T$:
\beq \label{zee}
\widehat Z(I_T, J_T; R_T ) =  \sum_\bigtriangleup \wp(I_T, J_T; y_k(R_T);  \text{split before}\, I_T-r_T) \text{ a.s.}
\eeq

\begin{prop} \label{last_unif} 
For some $\epsilon^{(3)}, \epsilon^{(4)}>0$ and $T$ large enough the following bounds hold
\beq \bea \label{really_almost}
& \wp(I_T, J_T; y_k(R_T);  \text{split before}\, I_T-r_T)\\
& \qquad \leq (\ln T)^{\epsilon^{(3)}} e^{-(\ln T)^{\epsilon^{(4)}}} y_k(R_T) e^{-\sqrt{2}(y_k(R_T)},
\eea \eeq
uniformly for all $k \in \bigtriangleup$ and $I_T, J_T$ as considered, almost surely. 
\end{prop}
The proof of this proposition is technical, and postponed to section \ref{proof_last_unif}.
We show how this provides the last piece for the proof of Theorem \ref{uniform_bound_lyons}. This is straightforward: by similar computations as in \eqref{getting_rid_r},
 \beq
\E\left[\sum_{k \in \bigtriangleup} y_k(R_T) e^{-\sqrt{2}y_k(R_T)}\right] \leq \kappa \cdot \sqrt{R_T} = 
\kappa \ln(T)^{1/2\delta},
\eeq
for large enough $\kappa>0$ and recalling that by definition $R_T = 40 (\ln T)^{1/\delta}$. This, together with \eqref{really_almost}  implies
\beq \bea
\E\left[\sum_{k \in \bigtriangleup} \wp(I_T, J_T; y_k(R_T);  \text{split before}\, I_T-r_T)\right]\leq \kappa (\ln T)^{\epsilon^{(3)}+1/2\delta} e^{-(\ln T)^{\epsilon^{(4)}}}.
\eea \eeq
Combining this with \eqref{bounding_stupid_one} we see that 
the claim of Theorem \ref{uniform_bound_lyons} holds with 
\beq 
\epsilon_1 \defi \max\left\{ \epsilon^{(1)}; \epsilon^{(3)}+1/2\delta\right\},
\eeq 
and
\beq 
\epsilon_2 \defi \min\left\{ \epsilon^{(2)}; \epsilon^{(4)}  \right\}.
\eeq

\hfill $\square$

\subsection{Proof of Proposition \ref{last_unif}} \label{proof_last_unif}
The claim is that 
\beq \bea \label{zee_zee}
& \wp(I_T, J_T; y_k(R_T);  \text{split before}\, I_T-r_T)\\
& \qquad \leq (\ln T)^{\epsilon^{(3)}} e^{-(\ln T)^{\epsilon^{(4)}}} y_k(R_T) e^{-\sqrt{2}y_k(R_T)},
\eea \eeq
holds uniformly for $k\in \bigtriangleup$. 
In order to prove this, we use a formula by Sawyer \cite{sawyer} concerning the {\it expected number of pairs of particles} ancestor branched in the interval $(0,I_T-r_T)$ and
whose paths satisfy certain localization conditions, say $T^{(1)}$ and $T^{(2)}$ respectively.
The expected number of such pairs is given by  
\beq\bea \label{sawyer_general}
& K e^{I_T} \int_0^{I_T-r_T} ds \cdot e^{J_T-s} \int d\mu_s(y) \PP\left[ x \in T^{(1)}_{(0,s)}\cap T^{(2)}_{(0,s)}  \mid x(s) = y\right] \times \\
& \times \PP\left[ x \in T^{(1)}_{(s,I_T)}  \mid x(s) = y\right] \times \PP\left[ x \in T_{(s, J_T)}^{(2)} \mid x(s) = y\right].
\eea\eeq
Here the probability $\PP$ is the law of a Brownian motion $x$, and $K=\sum_j p_j j(j-1)$ (with $\{p_j\}$ the offspring distribution). 
The time $s$ is the branching time of the common ancestor, and $\mu_s$ is the Gaussian measure with variance $s$. 
$T^{(\cdot)}_{(a,b)}$ denotes the condition on the path during the time interval $(a,b)$. 

A proof of this formula is given in \cite[p. 664 and 686]{sawyer}.
Sawyer counts the pairs of particles for the {\it same} time, whereas our case concerns particles for two different times: particle $\boldsymbol{(1)}$ at time $I_T$, and particle $\boldsymbol{(2)}$ at time $J_T$. 
The generalization of Sawyer's formula is however straightforward. 
The reader is referred to the intuitive construction of the formula provided by Bramson \cite[p. 564]{bramson}.

Dropping the condition $T^{(2)}$ in the first probability of \eqref{sawyer_general} yields a simpler bound:
\beq\bea \label{sawyer_general_two}
\eqref{sawyer_general} \leq & K e^{I_T} \int_0^{I_T-r_T} ds \cdot e^{J_T-s} \int d\mu_s(y) \PP\left[ x \in T^{(1)}_{(0,I_T)} \mid x(s) = y\right] \times \\
& \hspace{6cm} \times \PP\left[ x\in T_{(s, J_T)}^{(2)} \mid x(s) = y\right]\ .
\eea\eeq
Note that $\wp(I_T, J_T; y_k(R_T);  \text{split before}\, I_T-r_T)$ is by Markov inequality 
at most the expected number of pairs $\{\boldsymbol{(1), (2)}\}$ of particles which satisfy their 
respective localization conditions with the common ancestor branching before time $I_T-r_T$. 
By \eqref{sawyer_general_two}, it thus holds:
\beq \bea \label{huckleberry}
& \wp(I_T, J_T; y_k(R_T);  \text{split before}\, I_T-r_T) \\
& \qquad \leq K e^{I_T} \int_0^{I_T-r_T} ds \cdot e^{J_T-s} \int d\mu_s(y) \PP\left[ x \in T^{(1)}_{(0,I_T)} \mid x(s) = y\right] \times \\
& \hspace{8cm} \times \PP\left[ x\in T_{(s, J_T)}^{(2)} \mid x(s) = y\right]
\eea \eeq
with $T^{(1)}$ and $T^{(2)}$ being the shifted tubes defined in \eqref{cond_path_after_R_T} and \eqref{cond_path_after_R_T_two}.


The idea is now to bound the second probability appearing in \eqref{huckleberry} uniformly in $y$. 
This procedure has been introduced in Bramson \cite[Lemma 11]{bramson}, and proved useful also in \cite[Theorem 2.1]{abk_genealogies}. 

\begin{lem} \label{basta} It holds:
\beq \bea
& \PP\left[ x\in T_{(s, J_T)}^{(2)} \mid x(s) = y\right] \\
& \qquad \leq \Omega_T^2 e^{-(J_T-s)} \exp\left( -\sqrt{2}f_{\alpha, J}(R_T+s)-\frac{3}{2}\ln\left( \frac{J_T-s}{J_T}\right)-\frac{3}{2} \frac{s}{J_T}\ln J_T \right),
\eea \eeq
where $\Omega_T = O((\ln T)^{1/2\delta})$ as $T\uparrow \infty$.
\end{lem}

For the proof of  Lemma \ref{basta} some facts concerning the Brownian bridge are needed. 
 Denoting a standard Brownian motion by $x$, the Brownian bridge of length $t$ starting and ending at zero, is the Gaussian process
\beq 
\mathfrak z_t(s)\defi x(s)-\frac{s}{t} x(t), \qquad 0\leq s\leq t. 
\eeq
The Brownian bridge is a Markov process, and it has the property that $\mathfrak z_t(s), 0\leq s\leq t$
is independent of $x(t)$. This construction generalizes 
to the case where the endpoints of the bridge are $a, b\neq 0$; we denote by $\mathfrak z_t^{(a,b)}(s)$ such a process. The following is also well known: 
\beq \label{bridge_endpoints}
\mathfrak z_t^{(a,b)}(s) \stackrel{(d)}{=} \mathfrak z_t(s) + \left(1-\frac{s}{t} \right)a + \left(\frac{s}{t} \right)b, \qquad 0\leq s\leq t,
\eeq  
with equality holding in distribution. \\

We now recall \cite[Lemma 3.4]{abk_genealogies} which deals with probabilities that a Brownian bridge stays below linear functions; the proof is elementary and will not be given here.

\begin{lem} \label{bridge_cpam} Let $z_1, z_2 \geq 0$ and $r_1, r_2\geq 0$. Then for $t> r_1+r_2$, 
\beq \bea
& \PP\left[ \mathfrak z_t(s) \leq \left(1-\frac{s}{t} \right) z_1 +\frac{s}{t} z_2, \, r_1\leq s \leq t-r_2 \right] \\
& \hspace{4cm}\leq \frac{2}{t-r_1-r_2} \prod_{i=1,2} \left\{ z(r_i)+\sqrt{r_i} \right\},
\eea \eeq
where $z(r_1)\defi \left(1-\frac{r_1}{t} \right) z_1 +\frac{r_1}{t} z_2$ and
$z(r_2)\defi \frac{r_2}{t} z_1 + \left(1-\frac{r_2}{t}\right)z_2 $.
\end{lem}

\begin{proof}[Proof of Lemma \ref{basta}] We begin by first writing explicitly the underlying
conditions on the paths. For $f: \R_+ \to \R, t\mapsto f(t)$ a generic function, we denote  by $f^{S}(\cdot) \defi f(S+\cdot)$ its time-shift by $S>0$. We also shorten 
$y(s)\defi \sqrt{2}s-x(s)$, where $x(s)= y$ as in \eqref{huckleberry}, 
and $J_{T,s}\defi J_T-s$. We also set $\Omega_T \defi O(\ln \ln T)$.
By elementary manipulations one easily sees that 
\beq \bea
\PP\left[ x\in T_{(s, J_T)}^{(2)} \mid x(s) = y\right] = \PP\left[ \boldsymbol{(E)} \right],
\eea \eeq
where $\boldsymbol{(E)}$ is the event 
\beq \bea \label{huck_six} 
\boldsymbol{(E)} = \begin{cases} x(J_{T,s})\geq m(J_{T,s})+y(s)+\frac{3}{2\sqrt{2}}\ln\left(\frac{J_{T,s}}{J_T}\right)+ D+y_k(R_T) + \Omega_T & \hfill \boldsymbol{(E_1)}\\
F_1(t) \leq x(t) \leq F_2(t), \qquad 0 \leq t \leq J_{T,s}-r_T &  \hfill \boldsymbol{(E_2)}\\
\end{cases}
\eea\eeq
where $F_1, F_2$ are the entropic (resp. lower) envelopes of \eqref{cond_path_after_R_T} shifted by $s$:
\beq \bea \label{defi_F}
F_1(t) & \defi y_k(R_T)+y(s)+ \frac{t}{J_T} m(J_T)+\frac{3}{2\sqrt{2}}\frac{s}{J_T} \ln(J_T)
-f_{\alpha,J}^{R_T+s}(t)+\Omega_T, \\
F_2(t)  & \defi y_k(R_T)+y(s)+ \frac{t}{J_T} m(J_T)+\frac{3}{2\sqrt{2}}\frac{s}{J_T} \ln(J_T) 
-f_{\beta,J}^{R_T+s}(t)+\Omega_T,
\eea \eeq
with $\Omega_T=O(\ln \ln T)$. By the very same localizations, we also have a condition on $x(s)$. This reads
\beq \bea \label{condition_no_use}
x(s) &
\in \Big(-f_{\beta,J}^{R_T}(s); -f_{\alpha,J}^{R_T}(s)\Big)+ y_k(R_T)+\sqrt{2}s -\frac{3}{2\sqrt{2}}\frac{s}{J_T}\ln(J_T).
\eea\eeq
For later use, we reformulate \eqref{condition_no_use} into a condition on $y_k(R_T)+y(s)$,
namely:
\beq \label{cond_y_y}
y_k(R_T)+y(s) \in \Big(f_{\alpha,J}^{R_T}(s); f_{\beta,J}^{R_T}(s)\Big)+\frac{3}{2\sqrt{2}}\frac{s}{J_T}\ln(J_T).
\eeq

We now construct an event $\boldsymbol{(E')} \supsetneq \boldsymbol{(E)}$. First, we drop the condition that the Brownian path 
is required to stay {\it above} $F_2$. Second,  
we replace the condition on $F_1$ by the condition that the $x$-path remains, on the interval
$(0, J_{T,s}-r_T)$, {\it below} the line segment interpolating between $(0,F_1(0))$ 
and $(J_{T,s}, F_1(J_{T,s}))$, see Figure \ref{interpolation_fig} for a graphical representation. Precisely, 
we consider 
\beq \bea \label{huck_seven} 
\boldsymbol{(E')} = \begin{cases} x(J_{T,s})\geq m(J_{T,s})+y(s)+\frac{3}{2\sqrt{2}}\ln\left(\frac{J_{T,s}}{J_T}\right)+ D+y_k(R_T) + \Omega_T & \hfill \boldsymbol{(E'_1)}\\
x(t) \leq  \left(1-\frac{t}{J_{T,s}} \right) F_1(0) +  \frac{t}{J_{T,s}} F_1(J_{T,s})  \qquad 0 \leq t \leq J_{T,s}-r_T &  \hfill \boldsymbol{(E'_2)}\\
\end{cases}
\eea\eeq
By construction, 
\beq 
\PP\left[ \boldsymbol{(E)} \right] \leq \PP\left[ \boldsymbol{(E')}\right].
\eeq

\begin{figure}
\includegraphics[scale=0.4]{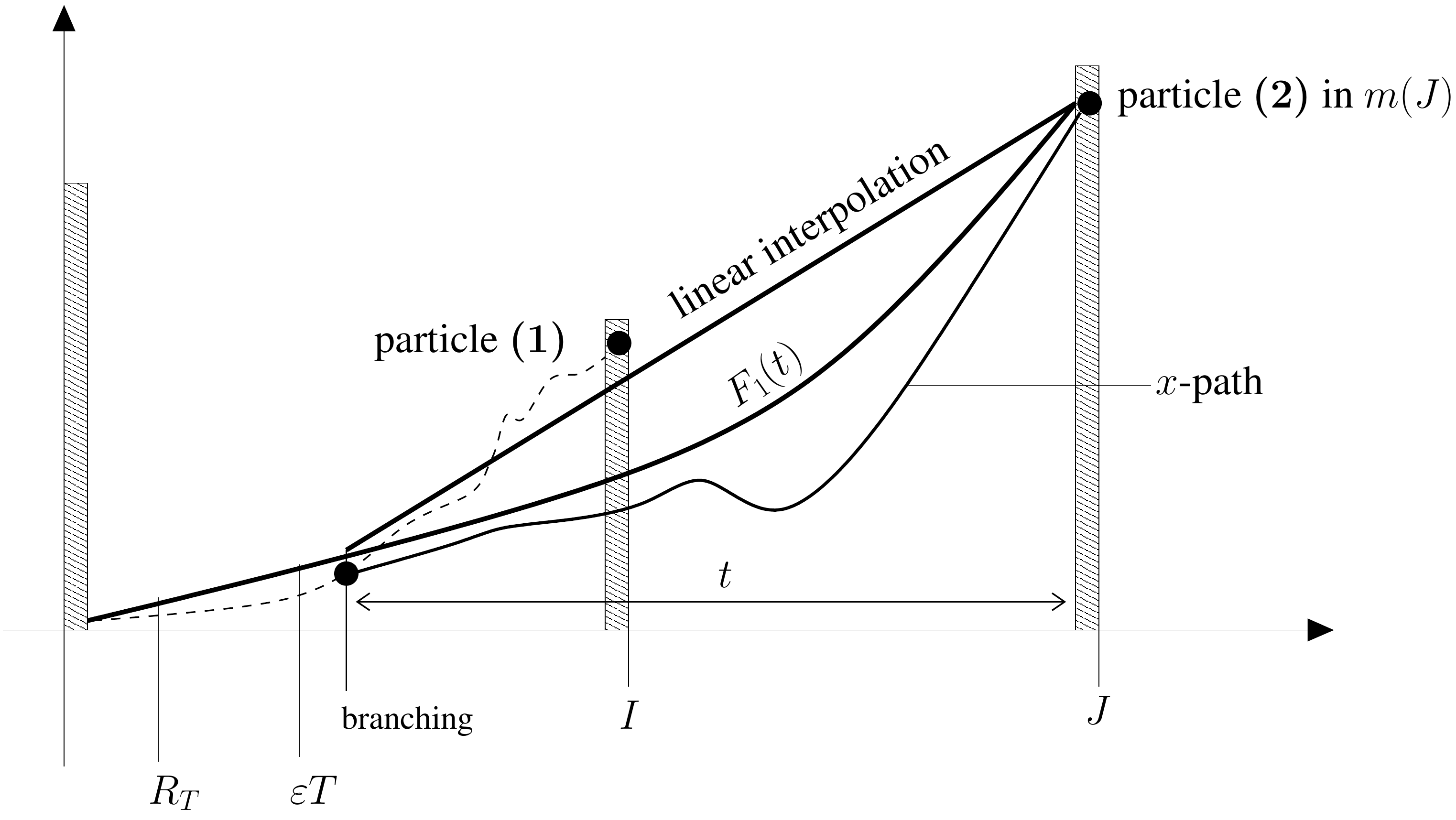}
\caption{The $x$-path stays below the linear interpolation}
\label{interpolation_fig}
\end{figure}

Let us put 
\beq \bea \label{defi_drift}
X(s, J_T) &\defi m(J_{T,s})+y(s)+\frac{3}{2\sqrt{2}}\ln\left(\frac{J_{T,s}}{J_T}\right)+ D+y_k(R_T) + \Omega_T \\
& = \sqrt{2} J_{T,s}-\frac{3}{2\sqrt{2}} \ln J_{T,s}+\left\{\frac{3}{2\sqrt{2}}\ln\left(\frac{J_{T,s}}{J_T}\right)+y(s)+y_k(R_T)  +\Omega_T\right\}.
\eea \eeq
We write
\beq \bea \label{p_e}
\PP\left[ \boldsymbol{(E')}\right] & = \int_{0}^{\infty} \PP\left[ \boldsymbol{(E'_2)} \big| x(J_{T,s}) = X(s, J_T)+X\right] \tilde \mu(dX),
\eea \eeq
where $\tilde \mu$ is a Gaussian with variance $J_{T,s}$ and mean $-X(s, J_T)$, i.e.
\beq 
\tilde \mu(dX) = \exp\left(-\frac{(X+X(s, J_T))^2}{2J_{T,s}} \right) \frac{dX}{\sqrt{2\pi J_{T,s}}}\,.
\eeq
We now make some observations concerning the Gaussian density and the conditional probability appearing in \eqref{p_e}. \\

For the Gaussian density, we recall that $J_{T,s} = J_T-s$ for $0\leq s \leq I_T-r_T\leq I_T$. Moreover, since $J_T-I_T > T^\xi$ and  $J_T \geq \vare T$, we see that 
\beq \label{uno}
- (1-\xi) \ln T - \ln \vare \leq \ln\left(\frac{J_{T,s}}{J_T}\right)  \leq 0.
\eeq
And, 
\beq \label{due}
y(s)+y_k(R_T)=o(J_{T,s}) \quad (T \uparrow \infty),
\eeq
by \eqref{cond_y_y}. 
Therefore, combining \eqref{uno} and \eqref{due} we have that $X(s, J_T) = \sqrt{2}J_{T,s}+ o(J_{T,s})$ as $T\uparrow \infty$.
The Gaussian density can thus be developed as follows
\beq \label{Gaussian_bound}
\tilde \mu(dX) =  J_{T,s}e^{-J_{T,s}} e^{-\sqrt{2} \Delta_T(s)} g_{T}(X) dX,
\eeq
where 
\beq 
\Delta_T(s) \defi y(s)+ \frac{3}{2\sqrt{2}}\ln\left(\frac{J_{T,s}}{J_T}\right)     +y_k(R_T),
\eeq
and
\beq
g_{T}(X) \defi \frac{e^{-X^2/2J_{T,s}}}{\sqrt{2\pi}}  e^{-\sqrt{2}(1+\omega_T)X}\left(1+\Omega_T \right),
\eeq
$\omega_T = o(1)$ as $T\uparrow \infty$, and $\Omega_T = O(\ln \ln T)$. \\

For the conditional probability appearing in \eqref{p_e}, we observe that conditioning on the event $\{x(J_{T,s}) = X\}$, turns the Brownian motion involved in the 
definition of $\boldsymbol{E'_2}$ into a Brownian bridge ending at the conditioning point. Precisely,  
\beq \bea
& \PP\left[ \boldsymbol{(E'_2)} \mid x(J_{T,s}) = X(s, J_T)+X \right] =\PP\left[ \boldsymbol{(E'')} \right],
\eea \eeq
where
\beq \bea
\boldsymbol{(E'')} & \defi \Big\{\forall_{0 \leq t \leq J_{T,s}-r_T}: \; \mathfrak z_{J_{T,s}}(t) \leq  \left(1-\frac{t}{J_{T,s}} \right) F_1(0) +  \frac{t}{J_{T,s}}\left( F_1(J_{T,s})-X(s, J_T)-X\right) \Big\}\\
& \quad = \Big\{\forall_{0 \leq t \leq J_{T,s}-r_T}: \; \mathfrak z_{J_{T,s}}(t) \leq  \left(1-\frac{t}{J_{T,s}} \right) F_1(0) +  \frac{t}{J_{T,s}}\left( \Omega_T-X\right) \Big\},
\eea \eeq
since by \eqref{defi_F} one has $F_1(J_{T,s})=\Omega_T= O(\ln \ln T)$. We
easily compute an upper bound to the probability of the $\boldsymbol{(E'')}$-event. By Lemma \ref{bridge_cpam}, putting there $z_1 \defi  F_1(0)$ and $z_2\defi \max\{\Omega_T-X; 0\}$), it holds:
\beq \bea
\PP\left[\boldsymbol{(E'')}\right] & \leq \frac{2}{J_{T,s}-r_T} F_1(0) \left(\frac{r_T}{J_{T,s}} F_1(0)
+ \left(1-\frac{r_T}{J_{T,s}} \right) \max\{\Omega_T-X; 0\} +\sqrt{r_T}\right).
\eea \eeq
Since $F_1(0)= y_k(R_T)+y(s)-f_{\beta, J}(R_T+s)\leq \Omega_T$ by the localization \eqref{due}, and 
$r_T \ll J_{T,s}= O(T)$, as $T\uparrow \infty$,
\beq \label{p_e_due}
\PP\left[\boldsymbol{(E'')}\right]\leq \frac{2\max\{\Omega_T-X; 0\}+\sqrt{r_T}}{J_{T,s}}.
\eeq
If we now plug the bounds \eqref{p_e_due} and \eqref{Gaussian_bound} into \eqref{p_e}, perform
the integral over $dX$, we immediately get that 
\beq \label{almost_done}
\PP\left[ \boldsymbol{(E')}\right] \leq \Omega_T^2 e^{-J_{T,s}-\sqrt{2}\Delta_T(s)},
\eeq
for some $\Omega_T  = O\left((\ln T)^{\epsilon^{(7)}}\right)$. By \eqref{cond_y_y} we may now bound $\Delta_T(s)$ from below, {\it uniformly} in $y(s)$: the upshot is 
\beq
\PP\left[ \boldsymbol{(E')}\right] \leq \Omega_T^2 e^{-J_{T,s}} \exp\left( -\sqrt{2}f_{\alpha, J}(R_T+s)-
\frac{3}{2}\ln\left( \frac{J_{T,s}}{J_T}\right)-\frac{3}{2} \frac{s}{J_T}\ln J_T \right).
\eeq
This is the uniform bound we were looking for and concludes the proof of Lemma \ref{basta}. 
\end{proof}

We finally give the 
\begin{proof}[Proof of Proposition \ref{last_unif}] Using the uniform bound provided by Lemma \ref{basta} in \eqref{huckleberry} and integrating over $\mu_s(dy)$ we  obtain
\beq \bea \label{huckleberry_two}
& \wp(I_T, J_T; y_k(R_T);  \text{split before}\, I_T-r_T) \leq \kappa \cdot \Omega_T \cdot e^{I_T} \cdot\PP\left[ x \in T^{(1)}_{(0,I_T)}\right] \times\\
& \qquad \times \int_0^{I_T-r_T} ds \cdot \exp\left( -\sqrt{2}f_{\alpha, J}(R_T+s)-\frac{3}{2}\ln\left( \frac{J_T-s}{J_T}\right)-\frac{3}{2} \frac{s}{J_T}\ln J_T \right).
\eea \eeq

The term $e^{I_T}\PP\left[ x \in T^{(1)}_{(0,I_T)}\right]$ can be handled by considerations 
similar to those in the proof of Lemma \ref{basta}. 
The condition $T^{(1)}_{(0,I_T)}$ gives rise to the event 
\beq \bea \label{huck_old} 
\begin{cases} x(I_T)\geq m(I_T)+D+y_k(R_T), & \\
F_2(t) \leq x(t) \leq F_1(t), & 0 \leq t \leq I_T-r_T.
\end{cases}
\eea\eeq
where
\beq \bea
F_1(t) & \defi y_k(R_T) + \frac{t}{I_T} m(I_T) - f_{\alpha,I}^{R_T}(t) + \Omega_T \\
F_2(t) & \defi y_k(R_T) +  \frac{t}{I_T} m(I_T) - f_{\beta,I}^{R_T}(t) + \Omega_T.
\eea\eeq
(For some $\Omega_T=O(\ln \ln T)$). 
In particular, the probability of the event is bounded by the probability that a Brownian motion
stays below the linear interpolation of the points $(0,F_1(0))$ and $(I_T-r_T, F_1(I_T-r_T))$ during the interval
of time $(0,I_T-r_T)$ intersected with the event $x(I_t-r_T)\geq F_2(I_T - r_T)$, that is:
\beq 
\bea
\PP\Bigg[
x(t)\leq \frac{t}{I_T-r_T} F_1(I_T-r_T) + \left(1 - \frac{t}{I_T-r_T} \right)  F_1(0), ~ \forall 0\leq t \leq I_T-r_T, \\ x(I_t-r_T)\geq F_2(I_T - r_T)
\Bigg]
\eea
\eeq
Subtracting $\frac{t}{I_T-r_T} x(I_T-r_T)$ and using the fact that  $x(I_t-r_T)\geq F_2(I_T - r_T)$, the above can be bounded above 
by $\PP\Big[x(I_t-r_T)\geq F_2(I_T - r_T) \Big]$ times the Brownian bridge probability:
\beq 
\bea
& \PP\Bigg[
\mathfrak z_{I_T-r_T}(t)\leq \frac{t}{I_T-r_T} (F_1(I_T-r_T)-F_2(I_T-r_T)) \\
& \hspace{6cm} + \left(1 - \frac{t}{I_T-r_T} \right)  F_1(0), ~ \forall 0\leq t \leq I_T-r_T\Bigg]\ .
\eea
\eeq
Now $F_1(I_T-r_T)-F_2(I_T-r_T)\leq \kappa R_T^\beta$, for some $\kappa>0$. Therefore the probability of the Brownian bridge can be bounded using Lemma \ref{bridge_cpam} by
\beq\label{crap}
\frac{2\kappa}{I_T-r_t} R_T^\beta~ F_1(0)=\frac{2\kappa}{I_T-r_t} R_T^\beta ~(y_k(R_T)+D - R_T^\alpha +\Omega_T)\ .
\eeq
Now, note that $m(I_T)-m(I_T-r_T)=\sqrt{2} ~ r_T+ o(1)$. Therefore, for some $\kappa>0$,
\beq
\bea
F_2(I_T - r_T) - m(I_T-r_T)\geq  y_k(R_T)+  \kappa ~ r_T\ .
\eea
\eeq
A standard Gaussian estimate thus yields for some $\epsilon >0$,
\beq\bea 
& \PP\Big[x(I_t-r_T) - m(I_T-r_T) \geq F_2(I_T - r_T) - m(I_T-r_T) \Big] \\
& \hspace{7cm} \leq \kappa (I_T-r_T) e^{ - \sqrt{2}y_k(R_T)} e^{ -(\ln T)^{\epsilon}}.
\eea\eeq 
A combination of the above equation and \eqref{crap} gives a bound of the desired form \eqref{really_almost}.

It thus remains to provide a simlar bounds for the integral in \eqref{huckleberry_two}. We first write 
\beq
\int_0^{I_T-r_T} = \int_0^{I_T/2}  + \int_{I_T/2}^{I_T-r_T}
\eeq
For the first integral, since $s\leq I_T/2$, we have 
\beq \bea 
\ln\left(\frac{J_{T}-s}{J_T} \right) & = \ln\left(1-\frac{s}{J_T} \right) \geq \ln\left(\frac{1}{2} \right)
\eea\eeq
hence, up to irrelevant numerical constant, the contribution of the first integral is at most 
\beq \label{first_contr}\bea 
& \Omega_T^2 \int_0^{I_T/2} ds  e^{-\sqrt{2} (R_T+s)^{\alpha}} \leq \Omega_T^2 \int_{R_T}^{\infty} ds e^{-\sqrt{2}s^{\alpha}} \leq e^{-\sqrt{2}R_T^{\epsilon^{(5)}}}
\eea \eeq
for some $\epsilon^{(5)}>0$ small enough. 
The contribution of the second integral is sub-exponentially small (in $T$). To see this, recall that $J_T-I_T>T^\xi$ and $s\in[I_T/2,I_T-r_T]$, thus for some $\kappa_1<0<\kappa_2$,
\beq \bea
\kappa_1 \ln T \leq \ln\left(\frac{J_T-s}{J_T} \right) & \leq \kappa_2 \ln T 
\eea\eeq
implying that the second integral is, for some $\kappa >0$, at most 
\beq \label{second_contr} \bea
& T^{\kappa}\int_{I_T/2}^{I_T-r_T}  e^{-\sqrt{2}f_{\alpha,J}(R_T+s)} ds
& \leq T^{\kappa} e^{- T^{\epsilon^{(6)}}} \leq e^{-T^{\epsilon^{(7)}}}
\eea \eeq
for some $\epsilon^{(6)}, \epsilon^{(7)}>0$. This is obviously much smaller than the first contribution \eqref{first_contr}.
Therefore, summing thus up,
\beq \bea
& \wp(I_T, J_T; y_k(R_T); \text{split before}\, I_T-r_T) \\
& \qquad \leq (\ln T)^{\epsilon^{(6)}} e^{-(\ln T)^{\epsilon^{(5)}}} y_k(R_T)e^{-\sqrt{2} y_k(R_T)}.
\eea \eeq
This concludes the proof of Proposition \ref{last_unif} by putting 
$\epsilon^{(3)} \defi \epsilon^{(6)}$ and $\epsilon^{(4)}\defi \epsilon^{(5)}$.

\end{proof}

{\bf Acknowledgments.} NK gratefully acknowledges discussions with B. Derrida and E. Brunet.

\end{document}